\documentclass[12pt]{amsart}

\usepackage{epsf,amssymb,times,overpic,amscd}
\usepackage{url}
\usepackage{hyperref}
\usepackage{microtype}

\usepackage{color}

\newtheorem{thm}{Theorem}[section]
\newtheorem{prop}[thm]{Proposition}
\newtheorem{lemma}[thm]{Lemma}
\newtheorem{cor}[thm]{Corollary}
\newtheorem{conj}[thm]{Conjecture}

\theoremstyle{definition}
\newtheorem{defn}[thm]{Definition}

\theoremstyle{remark}

\newtheorem{rmk}[thm]{Remark}

\newcommand{\R}{\mathbb{R}}
\newcommand{\Z}{\mathbb{Z}}

\newcommand{\N}{\mathbb{N}}
\renewcommand{\P}{\mathcal{P}}

\newcommand{\F}{\mathcal{F}}

\newcommand{\be}{\begin{enumerate}}
\newcommand{\ee}{\end{enumerate}}

\begin{document}
\title[Supporting broken book decompositions for contact forms]{On the existence of supporting broken book decompositions for contact forms in dimension $3$}
\author{Vincent Colin}
\address{V. Colin, Nantes Universit\'e, CNRS, Laboratoire de Math\'ematiques Jean Leray, LMJL, F-44000 Nantes, France}
\email{vincent.colin@univ-nantes.fr}
\urladdr{https://www.math.sciences.univ-nantes.fr/~vcolin/}

\author{Pierre Dehornoy}
\address{P. Dehornoy, Univ.\,Grenoble Alpes, CNRS, Institut Fourier, F-38000 Grenoble, France}
\email{pierre.dehornoy@univ-grenoble-alpes.fr}
\urladdr{http://www-fourier.ujf-grenoble.fr/~dehornop/}

\author{Ana Rechtman}
\address{A. Rechtman, Institut de Recherche Math\'ematique Avanc\'ee,
Universit\'e de Strasbourg,
7 rue Ren\'e Descartes,
67084 Strasbourg, France}
\email{rechtman@math.unistra.fr}
\urladdr{https://irma.math.unistra.fr/~rechtman/}

\date{This version: \today}

\keywords{open book decomposition, broken book decomposition, Birkhoff section, Reeb vector field, entropy, periodic orbit}

\makeatletter
\@namedef{subjclassname@1991}{Mathematics Subject Classification (2020)}
\makeatother
\subjclass{53E50, 57K33, 37C27, 37C35}

\begin{abstract}
We prove that in dimension $3$ every nondegenerate contact form is carried by a broken book decomposition. As an application we obtain that on a closed $3$-manifold, every nondegenerate Reeb vector field has either two or infinitely many periodic orbits, and two periodic orbits are possible only on the tight sphere or on a tight lens space.
Moreover we get that if $M$ is a closed oriented $3$-manifold that is not a graph manifold, for example a hyperbolic manifold, then every nondegenerate Reeb vector field on $M$ has positive topological entropy. 

\end{abstract}

\maketitle

\tableofcontents

\section{Introduction}

On a closed $3$-manifold $M$, the {\it Giroux correspondence} asserts that every contact structure $\xi$ is carried by {\it some} open book decomposition of $M$: there exists a Reeb vector field for~$\xi$ transverse to the interior of the pages and tangent to the binding~\cite{Gi}. 
The dynamics of this specific Reeb vector field is then captured by its first-return map on a page, which is a flux-zero area-preserving diffeomorphism of a compact surface, a much simplified data.
When one is interested in the dynamics of a {\it given} Reeb vector field this Giroux correspondence is quite unsatisfactory---though there are ways to transfer some properties of an adapted Reeb vector field to every other one through contact homology techniques \cite{CH, ACH}---and the question one can ask is: Is every Reeb vector field carried by  some (rational) open book decomposition? 
Equivalently, does every Reeb vector field admit a Birkhoff section? 

We give here a positive answer for the generic class of nondegenerate Reeb vector fields and the extended class of {\it broken book decompositions} (see Definitions~\ref{defn-sections} -- \ref{defn-carry} for details).

\begin{thm}\label{thm: main}
On a closed 3-manifold, there is an open $C^1$-neighborhood of the set of nondegenerate Reeb vector fields  such that every Reeb vector field in this neighborhood is $\partial$-strongly carried by a broken book decomposition.
\end{thm}

A contact form and the corresponding Reeb vector field are {\it nondegenerate} if all the periodic orbits of the Reeb vector field are nondegenerate, namely the eigenvalues of the differentials of the return maps on small discs tranverse to the periodic orbits are all different from one (even when the orbit is travelled several times). 
The nondegeneracy condition is generic for Reeb vector fields~\cite[Lemma~7.1]{CH}. 

For a vector field $R$ on a 3-manifold, a {\it $R$-section} is a surface with boundary whose interior is embedded and transverse to~$R$ and whose boundary is immersed and composed of periodic orbits (it is sometimes called a {partial section for~$R$}). 
A $R$-section is {\it $\partial$-strong} when the linearized flow of $R$ along every boundary orbit is transversal to the boundary of the surface.

A {\it Birkhoff section} for~$R$ is a $R$-section that must also intersect all orbits of $R$ within bounded time, so that there is a well-defined first-return map in the interior of the surface. 
These surfaces are also known as rational global surfaces of section. 
A given Birkhoff section induces a rational open book decomposition of the manifold (where we do not require the usual condition of compatibility of orientations along the binding, see Section~\ref{ssec-contact}), all of whose pages are also Birkhoff sections. 
For this reason we refer to the boundary of a Birkhoff section as the {\it binding}. 

Broken book decompositions are generalisations of Birkhoff sections and rational open book decompositions, reminiscent of finite energy foliations constructed by Hofer, Wyszocki and Zehnder for nondegenerate Reeb vector fields on $\mathbb{S}^3$~\cite{HWZ}. 
In a broken book decomposition we allow the binding to have {\it broken} components, in addition to {\it radial} ones modelled on the classical open book case. 
The complement of the binding is foliated by surfaces that are relatively compact in $M$ and whose closures in $M$ are called the {\it pages}. 
A radial component of the binding has a tubular neighborhood in which the pages of the broken book induce a radial foliation by annuli (as in Figure~\ref{figure: radialcomponent}). 
The foliation in a tubular neighborhood of a broken component has sectors that are radially foliated by annuli and sectors that are foliated by surfaces that transversally look like hyperbolas (see Figures~\ref{figure: brokencomponent0} and~\ref{figure: brokencomponent1}). 
For the broken books we construct in this paper, each broken component has either two or four sectors of each type, depending on whether the component is of positive or negative hyperbolic type. 

A broken book decomposition {\it carries}---or {\it supports}---a vector field~$R$ if the binding is composed of periodic orbits, while the other orbits are transverse to the interior of the pages. 
In particular the pages are $R$-sections whose boundary is contained in the binding and whose interiors give in general a non-trivial foliation of the complement of the binding, as opposed to the genuine open book case.
The broken book {\it $\partial$-strongly carries} the vector field~$R$ if moreover the pages are $\partial$-strong $R$-sections. 

\medskip 

In the proof of Theorem \ref{thm: main}, we construct a supporting broken book decomposition for any fixed nondegenerate Reeb vector field on a $3$-manifold~$M$ from a cover of $M$ by pseudo-holomorphic curves, given by the non-triviality of the $U$-map in embedded contact homology. 
Using a construction of Fried that glues and resolves such curves~\cite{Fr}, the projected pseudo-holomorphic curves are converted into $R$-sections. 
We can then extract from this collection a finite complete system of disjoint $R$-sections to the Reeb vector field, meaning that their union intersects every orbit. 
We then complete the broken book decomposition in the complement of the chosen $R$-sections. 
The novelty in our approach is to combine pseudo-holomorphic curves with Fried's techniques. 

 It is worth noticing that the binding of the broken book produced by Theorem \ref{thm: main} for a contact form $\lambda$ can be taken to have total action less than or equal to the {\it spectrum} of any class $\sigma \in ECH(M,\lambda)$ with $U(\sigma)\neq 0$, see Section \ref{sec: ECH} for a definition. 

We believe that the notion of a (degenerate) broken book decomposition is interesting in its own right. Near the binding, the broken book foliation looks like the mapping torus of a transverse foliation of Le Calvez~\cite{LeC} and our study could also be seen as a first step towards generalising Le Calvez' theory to vector fields in three dimensions.

\medskip
 
Birkhoff sections are often useful for understanding the dynamics of a given vector field. 
We hope that broken book decompositions will also help answering dynamical questions. 
In this direction, we give two applications of Theorem~\ref{thm: main}.

Weinstein conjectured in 1979 that a Reeb vector field on a closed \linebreak 3-manifold always has at least one periodic orbit \cite{Wei79}. The conjecture was proved in full generality by Taubes using Seiberg-Witten Floer homology \cite{Tau}. 
It is also a consequence of the $U$-map property we use here, and it is no surprise that our result indeed implies the existence of the binding periodic orbits. Taubes' result was then improved by Cristofaro-Gardiner and Hutchings \cite{C-GH} who proved that every Reeb vector field on a closed 3-manifold has at least two periodic orbits, following a work of Ginzburg, Hein, Hryniewicz and Macarini on $\mathbb{S}^3$~\cite{GHHM}. 
It is now moreover conjectured that a Reeb vector field has either two or infinitely many periodic orbits. 
The existence of infinitely many periodic orbits has been established under some hypothesis (see the survey \cite{GG}) and Irie proved that it is generic~\cite{I}. Here we extend a recent result of Cristofaro-Gardiner, Hutchings and Pomerleano, originally obtained for {\it torsion contact structures} $\xi$ (that is, satisfying $c_1(\xi) \in \mathrm{Tor}(H^2(M,\Z))$)~\cite{CHP} and prove the conjecture for an open neighborhood of nondegenerate Reeb vector fields.

\begin{thm}\label{thm: infinite} If $M$ is a closed oriented $3$-manifold that is not the sphere or a lens space, then there is an open $C^1$-neighborhood of the set of 
nondegenerate Reeb vector fields on~$M$ such that every Reeb vector field in this neighborhood has infinitely many simple periodic orbits. 
In the case of the sphere or a lens space, there is an open $C^1$-neighborhood of the set of nondegenerate Reeb vector fields such that every Reeb vector field in this neighborhood has either two or infinitely many periodic orbits. 
\end{thm}

We point out that the cases where Reeb vector fields have exactly two nondegenerate periodic orbits are well-understood: they exist only on the sphere or on lens spaces, both periodic orbits are elliptic and are the core circles of a genus one Heegaard splitting of the manifold~\cite{HT}. In these cases, the contact structure has to be tight, since a nondegenerate Reeb vector field of an overtwisted contact structure always has a hyperbolic periodic orbit, see for example~\cite[Theorem~8.9]{HK}. 

Beyond the number of periodic orbits, the study of the topological entropy of Reeb vector fields started with the work of Macarini and Schlenk \cite{MS} and has been continued by Alves \cite{ACH,A1}. 
We recall that topological entropy measures the complexity of a flow by computing the growth of the number of ``different'' orbits. 
For flows in dimension 3, if the topological entropy is positive then the number of periodic orbits of period less than a positive number~$L$ grows exponentially with~$L$. 
As another application of Theorem~\ref{thm: main} we get 

\begin{thm}\label{thm: entropy} If $M$ is a closed oriented $3$-manifold that is not a graph manifold, then there is an open $C^1$-neighborhood of the set of 
nondegenerate Reeb vector fields on~$M$ such that every Reeb vector field in this neighborhood has positive topological entropy.
\end{thm}

Theorems~\ref{thm: infinite} and~\ref{thm: entropy} are obtained by analysing the broken binding components of the broken book decomposition. 
Indeed, in the nondegenerate context, a broken component of the binding has to be a hyperbolic periodic orbit. 
As such it has stable and unstable manifolds~\cite{Sm}. 
We can prove that there are heteroclinic cycles between these periodic orbits. Here a {\it heteroclinic cycle} is a finite sequence of orbits of the flow $\gamma_0,\dots,\gamma_n=\gamma_0$ such that every $\gamma_i$ is forward and backward asymptotic to a hyperbolic periodic orbit and the forward limit orbit of $\gamma_i$ is equal to the backward limit of $\gamma_{i+1}$, for $i=0,\dots n{-}1$. 
A {\it homoclinic orbit} is a heteroclinic cycle with $n=0$: an orbit that is forward and backward asymptotic to a hyperbolic periodic orbit. 
If there are no broken components, then we have a rational open book decomposition and Theorems~\ref{thm: infinite} and~\ref{thm: entropy} are deduced from an analysis of its monodromy.
In particular, we obtain

\begin{thm}\label{thm: homoclinic} If $M$ is a closed oriented $3$-manifold, there is an open $C^1$-neighborhood of the set of strongly nondegenerate Reeb vector fields on~$M$ without homoclinic orbits such that every Reeb vector field in this neighborhood is carried by a rational open book decomposition, or, equivalently has a Birkhoff section.
\end{thm}

A vector field is {\it strongly nondegenerate} if it is nondegenerate and the intersections of the stable and unstable manifolds of the hyperbolic orbits are transverse. By a result of Katok, a strongly nondegenerate vector field with a homoclinic orbit has positive topological entropy~\cite{Katok}, thus Theorem \ref{thm: homoclinic} implies that a strongly nondegenerate Reeb vector field whose topological entropy is zero is carried by a rational open book decomposition.

\medskip

Our techniques, combined with Fried's construction~\cite{Fr}, also allow to establish the existence of a carrying rational open book decomposition when there is only one broken component in the binding. 
Supported by these constructions, we make the optimistic Conjecture~\ref{conj-obd} that broken book decompositions can be transformed into rational open book decompositions, and thus that nondegenerate Reeb vector fields always admit Birkhoff sections.

\medskip

The results presented in this introduction are first proved for nondegenerate Reeb vector fields and extended to $C^1$-neighborhoods at the end of the paper. 
In Section~\ref{sec-bob} we define broken book decompositions and how they carry vector fields or contact forms. 
The existence of broken book decompositions is established in Section~\ref{sec-ech}, in particular we give a proof of Theorem~\ref{thm: main} for nondegenerate Reeb vector fields. 
The applications of this theorem are discussed in Section~\ref{sec-app}, which is divided in seven parts. 
After two sections containing preliminaries, in Section~\ref{ssec-int} we study the invariant stable and unstable manifolds of the broken components of the binding. 
In particular we prove that they have to intersect. 
Then we give a proof of Theorem~\ref{thm: infinite} for strongly nondegenerate Reeb vector fields in order to explain the main ideas in Section~\ref{ssec-strongly}. 
Passing from strongly nondegenerate to nondegenerate Reeb vector fields in the proof of Theorem~\ref{thm: infinite} is rather technical, so this proof is in Section~\ref{ssec-nondeg}, which also contains the proof of Theorem~\ref{thm: entropy}.
In Section~\ref{ssec-other} we prove other applications of the existence of broken book decompositions, in particular Theorem~\ref{thm: homoclinic} for nondegenerate Reeb vector fields and explain a construction that can be applied to get rid of broken components of the binding. 
Finally, in Section~\ref{ssec-improving}, we extend the proofs of Theorems~\ref{thm: main}, \ref{thm: infinite}, \ref{thm: entropy} and~\ref{thm: homoclinic} to an open neighborhood of the set of nondenegerate Reeb vector fields in the $C^1$-topology.

\medskip
\noindent{\bf Acknowledgements:} We thank Oliver Edtmair, Umberto Hryniewicz, Michael Hutchings, Rohil Prasad and the anonymous referees for useful exchanges and suggestions. 
 
V. Colin thanks the ANR Quantact for its support, P. Dehornoy thanks the ANR projects IdEx UGA and Gromeov for their supports, and A. Rechtman thanks the IdEx Unistra, Investments for the future program of the French Government. This project started during the Matrix event ``Dynamics, Foliations and Geometry in dimension $3$" held at Monash University in 2018. We thank these institutions for their support.

\section{Reeb vector fields and broken book decompositions}\label{sec-bob}

This section contains the definitions underlying Theorem~\ref{thm: main}.

\subsection{Sections and $\partial$-strong sections}\label{ssec-sections}

For a smooth vector field $R$ on a closed 3-manifold~$M$, we denote by $(\phi_R^t)_{t\in\R}$ its flow. 

\begin{defn}\label{defn-sections}
A {\it $R$-section} is an immersed surface $S$ in $M$ whose interior is embedded and transverse to $R$ and whose boundary covers periodic orbits of $R$.
A $R$-section $S$ is a {\it Birkhoff section} if it intersects all the orbits of~$R$ in bounded time: there exists $T>0$ such that for all $x\in M$ the orbit segment~$\phi_R^{[0,T]}(x)$ intersects $S$.
\end{defn}

Given a periodic orbit $\gamma$ of $R$ we denote by $\Sigma_\gamma$ the unit normal bundle~$(TM_\gamma/T\gamma)/\R_+$ to~$\gamma$ and by~$M_\gamma$ the normal blow-up of~$M$ along~$\gamma$, that is the manifold~$(M\setminus\gamma)\cup \Sigma_\gamma$.
The vector field~$R$ being smooth, it extends to a vector field~$R_\gamma$ on the torus~$\Sigma_\gamma$---and tangent to it---, and hence to a vector field on~$M_\gamma$ and tangent to its boundary. 
We abuse notation and still denote this extension by~$R$.
If $S$ is a $R$-section with~$\gamma \in \partial S$, we denote by~$\partial_\gamma S$ is extension to~$\Sigma_\gamma$ (see Figure~\ref{figure: dstrong}).

\begin{defn}\label{defn-strong}
A $R$-section $S$ is {\it $\partial$-strong} if, for every boundary orbit $\gamma$ of~$S$, the extension~$\partial_\gamma S$ is a collection of embedded curves in~$\Sigma_\gamma$ that are transverse to the extended vector field~$R$.
If $S$ is a Birkhoff section, $S$ is {\it $\partial$-strong} if moreover $\partial_\gamma S$ defines a transverse section for~$R$ on~$\Sigma_\gamma$.
\end{defn}

\begin{figure}[ht]
\centering
\includegraphics[width=.5\textwidth]{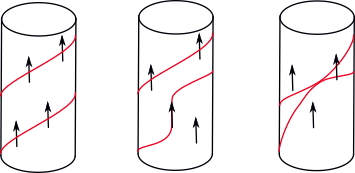}
\caption{The unit normal bundle~$\Sigma_\gamma$ to a periodic orbit~$\gamma$ of a vector field and the extension~$\partial_\gamma S$ of a $R$-section~$S$ to it~(in red). On the left $S$ is $\partial$-strong, while it is not in the center ($\partial_\gamma S$ is not transverse to $R$ on $\Sigma_\gamma$) nor on the right ($\partial_\gamma S$ is not embedded in~$\Sigma_\gamma$).}\label{figure: dstrong}
\end{figure}

\subsection{Contact forms, Reeb vector fields and open books}\label{ssec-contact}

Recall that on a 3-manifold, a {\it contact form} is a non-vanishing 1-form~$\lambda$ such that $\lambda\wedge d\lambda$ is a volume form. 
Contact forms exist on every orientable 3-manifold. 
The {\it Reeb vector field} associated to a contact form~$\lambda$ is the unique vector field~$R_\lambda$ satisfying
\[ i_{R_\lambda}\lambda=1\mathrm{\quad and\quad} i_{R_\lambda}d\lambda=0.\]
The vector field~$R_\lambda$ does not vanish and it preserves the volume~$\lambda\wedge d\lambda$. 
It is also transverse to the so-called \emph{contact structure} $\ker \lambda$. 

Recall that a {\it rational open book decomposition} of a closed 3-manifold~$M$ is a pair $(K,\F)$ where $K$ is an oriented link called the {\it binding} of the open book and $\F$ is a foliation of $M\setminus K$ defined by a fibration of $M\setminus K$ over~$\mathbb{S}^1$. Near every component $k$ of $K$ the foliation is as in Figure \ref{figure: radialcomponent}: in local cylindrical coordinates~$(r, \theta, z)$ the leaves are given by $\theta=$ cst.
A {\it page} of the open book is the closure of a leaf of $\F$ which is obtained by its union with $K$. 
In this context the pages are embedded in their interior, but only immersed along their boundary.
The adjective {\it rational} is dropped when moreover the pages are embedded along their boundary. 
So in an {\it open book decomposition} each page appears exactly once along each component of the binding. 
In both cases we say that $k$ is radial with respect to $\F$.

\begin{figure}[ht]
\centering
\includegraphics[width=.4\textwidth]{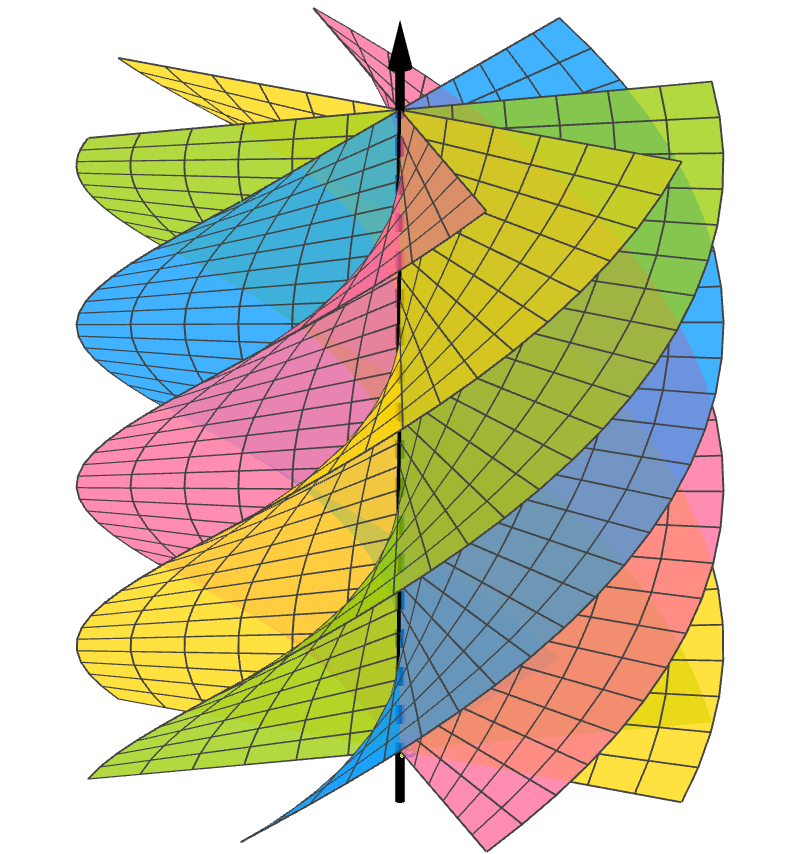}
\caption{A binding component of a rational open book where every page wraps twice along the binding component. The same picture appears in a neighborhood of a radial binding component of a broken book.}
\label{figure: radialcomponent}
\end{figure}

Given a contact form~$\lambda$ on~$M$, a rational open book decomposition~$(K, \F)$ of~$M$ {\it carries}~$\lambda$ if its Reeb vector field~$R_\lambda$ is tangent to the binding~$K$ and positively transverse to the interior of the pages.
Compared with the standard definition of Giroux \cite{Gi}, we de not assume here that the binding is positively tangent to the Reeb vector field with respect to the orientation induced by the pages.

In the previous context, any page~$S$ of the open book decomposition is a Birkhoff section for~$R_\lambda$ and the dynamics of~$R_\lambda$ may be studied through the first-return map~$f_S:S\to S$ induced by~$R_\lambda$. 
In particular, $f_S$ preserves the area-form $d\lambda|_{TS}$ in the interior of $S$. 
Moreover, $f_S$ is a {\it flux-zero} homeomorphism as in Definition~\ref{defn: fluxzero}.

\subsection{Degenerate broken book decompositions}\label{sec: degenerate}
We generalise open book decompositions
 by allowing another behaviour near the binding, namely {\it broken} components.
 While similar, the definition is more general than the one of {transverse foliations} proposed by Hryniewicz and Salam\~ao~\cite{HryS} and than the one of finite energy foliations of Hofer, Wysocki and Zehnder \cite{HWZ}. 
Although we mostly consider (nondegenerate) broken book decompositions, we start with a general definition. 

\begin{defn}\label{defn-degbob}
A {\it degenerate broken book decomposition} of a closed\linebreak $3$-manifold $M$ is a pair $(K,\mathcal{F})$ such that:
\begin{itemize}
\item $K$ is a link, called the {\it binding}; 
\item $\mathcal{F}$ is a cooriented foliation of $M\setminus K$ such that each leaf $S$ of $\mathcal{F}$ is properly embedded in $M\setminus K$ and admits a compactification $\overline{S}$ in~$M$ which is a compact surface, called a {\it page}, whose boundary is contained in $K$; {\it radial} and {\it broken} binding components respectively; a component $k_ r$ of $K$ is {radial} if~$\F$ foliates a neighborhood of $k_r$ by annuli all having exactly one boundary component on $k_r$. 
\end{itemize}
\end{defn}
 For a degenerate broken book, the pages are not assumed to be transversally smooth along the binding. 
 
In a neighborhood of a radial binding component, a degenerate broken book is similar to a rational open book. 
If~$K_b$ is empty the degenerate broken book decomposition is a rational open book decomposition. 

For $k_b$ a broken binding component and $U$ a tubular neighborhood of~$k_b$, there are two types of leaves of 
$\F\cap U$: either the boundary contains~$k_b$, or it does not. 
In the first case we speak of a {\it radial leaf}, and in the second of a {\it hyperbolic leaf}. 
Containing~$k_b$ in the boundary being a closed condition, the set of radial leaves is closed. 
The local picture is then the following: around~$k_b$ alternate {\it radial sectors} where the foliation is made of radial leaves and {\it hyperbolic sectors} where the foliation looks transversally (that is, when we intersect it with a disc transverse to~$k_b$) like hyperbolas, see Figure~\ref{figure: brokencomponent0}. Observe that the definition of radial and hyperbolic leaves applies only for the intersection of a leaf with a neighborhood of $k_b$.

\begin{figure}[ht]
\centering
\includegraphics[width=.3\textwidth]{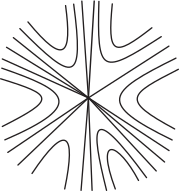}
\caption{A transversal view of a degenerate broken book decomposition near a broken binding component. There are six radial sectors that alternate with six hyperbolic sectors. Note that the radial sector at 2 o'clock consists of one leaf only.
 }\label{figure: brokencomponent0}
\end{figure}

\begin{defn}\label{defn-degcarry}
A vector field~$R$ is {\it carried} by a degenerate broken book decomposition $(K,\F)$ if it is tangent to $K$ and positively transverse to the leaves of $\F$.
A contact form $\lambda$ is {\it carried} by a degenerate broken book decomposition~$(K,\F)$ if its Reeb vector field $R_\lambda$ is carried by~$(K, \F)$. 
\end{defn}

By definition, the pages of a degenerate broken book decomposition carrying a contact form $\lambda$ are $R_\lambda$-sections. 
Thus, roughly speaking, a degenerate broken book decomposition carrying a Reeb vector field can be understood as a foliation by $R_\lambda$-sections, whereas a rational open book is a foliation by Birkhoff sections. 

If a degenerate broken book decomposition $(K, \F)$ carries a contact form~$\lambda$, its binding~$K$ is nonempty, for otherwise the vector field~$R_\lambda$ would be transverse to a compact surface with empty boundary, in contradiction with Stokes' theorem. 

Note that, as with rational open books in the previous section, we do not require that the orientation of the binding coming from the cooriented pages coincides with the orientation given by $R_\lambda$, as it is often required in the classical open book case. 
Actually, there is not even a preferred orientation at a broken binding component since different pages coming to the binding give different orientations depending on the sector they locally stand in. 

\subsection{Broken book decompositions}
Since the vector fields we consider in this article are nondegenerate, we can be a bit more restrictive concerning the broken books we consider. 
First the pages are assumed to be smooth. 
Next, recall that the periodic orbits of a nondegenerate Reeb vector field~$R_\lambda$ are of three types: elliptic, positive hyperbolic and negative hyperbolic, depending on whether the linearized first-return map on a small disc transverse to the periodic orbit is conjugated to an irrational rotation, has real positive eigenvalues, or has real negative eigenvalues.
If a degenerate broken book~$(K, \F)$ supports a nondegenerate contact form $\lambda$, a radial component of~$K$ (that is, one in~$K_r$) can be an elliptic or a hyperbolic periodic orbit of~$R_\lambda$; while a broken component of $K$ (one in~$K_b$) is necessarily a hyperbolic periodic orbit of $R_\lambda$. 
Indeed, if there was an elliptic orbit in~$K_b$, the vector field~$R_\lambda$ would fail to be transverse to~$\F$ in the sectors where the leaves of~$\F$ are hyperbolic-like.

Moreover, near each point of a broken component of the binding, the foliation has locally four hyperbolic sectors, separated by four radial sectors (as in Figure~\ref{figure: brokencomponent1}). 
This is due to the fact that the hyperbolic sectors cannot be transverse to the vector field if they contain no eigendirection of the broken orbit. 
Thus there cannot be more such sectors than eigendirections, and a given sector may not contain more than one eigendirection. 

\begin{defn}\label{defn-bob} 
A {\it broken book decomposition} (or \emph{broken book} for short) is a degenerate broken book decomposition whose pages are smooth and whose broken binding components locally have four hyperbolic sectors, separated by four radial sectors.
\end{defn}

Since a broken component of the binding is a hyperbolic periodic orbit, it can be positive or negative. If positive, the monodromy of~$\F$ along this orbit is the identity; and if negative the monodromy of $\mathcal{F}$ is a $\pi$-rotation, implying that the four local sectors correspond to two global sectors.

\begin{figure}[ht]
\centering
\includegraphics[width=.27\textwidth]{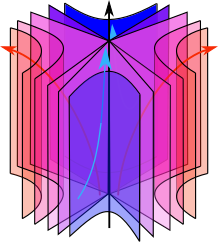}
\hspace{3mm}
\includegraphics[width=.3\textwidth]{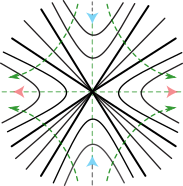}
\hspace{5mm}
\includegraphics[width=.3\textwidth]{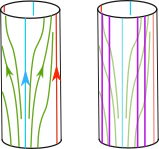}
\caption{A broken book decomposition $\partial$-strongly carrying a nondegenerate contact form~$\lambda$ in a neighborhood of a broken binding component. 
On the left the binding component (black) is shown, as well as some local pages of the broken book. 
Two unstable orbits (red) of the Reeb vector field~$R_\lambda$ and two stable orbits (blue) are also shown. 
In the center a transversal view, that is the intersection of the broken book with a disc transverse to~$R_\lambda$. 
The rigid pages are bolded. 
Some orbits of~$R_\lambda$ are represented with green dotted lines. 
In particular the four local stable/unstable manifolds lie in four different hyperbolic sectors.
On the right two pictures of the blow-up of the binding component with the extension of~$R_\lambda$ (first picture) and with the extension of the rigid pages (purple, right-most picture). 
The extensions of the pages are transverse to the extended vector field, hence the term $\partial$-strong.
 }\label{figure: brokencomponent1}
\end{figure}

\medskip 

The definition of a contact form (or a vector field) carried by a degenerate broken book extends to the nondegenerate case. 
However it is convenient to enforce some transversality on the binding components.

\begin{defn}\label{defn-carry}
A nondegenerate vector field~$R$ is {\it $\partial$-strongly carried} by a broken book decomposition~$(K, \F)$ if~$R$ is tangent to~$K$ and positively transverse to the pages of~$\F$, and if every page of~$\F$ is a $\partial$-strong $R$-section. 
A nondenegerate contact form~$\lambda$ is {\it $\partial$-strongly carried} by a broken book decomposition~$(K, \F)$ if its Reeb vector field~$R_\lambda$ is $\partial$-strongly carried by~$(K, \F)$.  
\end{defn}

 \medskip

We end this section with some general properties of broken book decompositions. 
When the broken binding is nonempty, we can distinguish different types of pages as follows. 

\begin{defn}\label{defn-rigid} A page $S$ of a broken book $(K, \F)$ that belongs to the interior of a 1-parameter family of pages of the form~$S\times[0,1]/\sim$, where $(y,t)\sim (y,t')$ for every $y\in \partial S$ and $t,t'\in [0,1]$, is called {\it regular}. On the other hand, a page that is not in the interior of such a 1-parameter family is called {\it rigid}. 
\end{defn}

A rigid leave must have at least one boundary component in the broken binding near which it  locally coincides with an annulus that separates a radial sector from a 
hyperbolic sector. Since there are finitely many such locally separating annuli, we get that there are only finitely many rigid leaves.

In the center of Figure~\ref{figure: brokencomponent1} is depicted a neighborhood of a broken binding component where the rigid leaves are bolded: these are the leaves that are the limit of a 1-parameter family of radial leaves. 
On the other side, the hyperbolic leaves break on them. 

For $(K, \F)$ a broken book decomposition with nonempty broken binding, in the complement of the set $\mathcal{R}$ of rigid pages the foliation 
$\mathcal{F}$  is a fibration over $\R$.

\medskip 

At this point, it is not clear to us whether a broken book decomposition supporting a contact form always has nonempty radial binding $K_r$.
We can prove it is the case when every regular page is either a disk or an annulus; this result will be used in the proof of Lemma~\ref{lemma: full}.

\begin{lemma}\label{lemma: radial nonempty}
A broken book decomposition supporting a contact form $\lambda$ and whose regular pages are all disks or annuli has nonempty radial binding.
\end{lemma}
\begin{proof} We argue by contradiction and assume that the radial binding is empty.
We start from a regular page $ P_1$, which is a disk or an annulus, having one boundary component on a broken binding component $h_1$.
The circle $h_1$ is a hyperbolic orbit for the Reeb flow $R_\lambda$ and its stable and unstable manifolds decompose a small neighborhood of $h_1$ into sectors. 
We consider a regular page $P_2$ that leaves $h_1$ in a sector adjacent to the one containing $P_1$. 
There are two possibilities: if $h_1$ is positive, then we have two choices for the sector and we pick $P_2$ in any of them and glue $P_2$ to $P_1$ along $h_1$; if $h_1$ is negative, there is only one choice of adjacent sector for $P_2$, but there are two ways to glue $P_2$ to $P_1$ and we pick any of them.
If $P_1$ and $P_2$ were both disks, we would get after gluing an immersed sphere (with corner along~$h_1$) that is positively transverse to the Reeb flow away from $h_1$, and obtain a contradiction using Stokes' theorem.
If one of them is an annulus, say $P_2$, we can consider its other boundary component, going to a binding component~$h_2$.
If $h_2=h_1$ and if $P_2$ arrives and leaves $h_1$ in two adjacent sectors then it gives an immersed torus positively transverse to the Reeb flow (except along $h_1$), again a contradiction.
If otherwise, we take a regular surface $P_3$ in the adjacent sector to the one locally containing $P_2$ near $h_2$.

Iterating this process, since there are only finitely many binding components and sectors, we have to either end with disks on both ends and find an immersed sphere positively transverse to $R_\lambda$ away from the binding components, or a cycle of surfaces $P_i,\dots,P_{i+k}$ forming an immersed torus after gluing and still positively transverse to $R_\lambda$ away from the binding components: in both cases a contradiction.
\end{proof}


\section{Construction of a broken book decomposition from ECH theory}\label{sec-ech}

In this section we give a proof of Theorem~\ref{thm: main} in the case of nondegenerate Reeb vector fields: for every nondegenerate contact form $\lambda$ on a closed oriented 3-manifold~$M$, we build a broken book decomposition that carries~$\lambda$. 

We first use embedded contact homology, in particular the\linebreak $U$-map, to obtain for every point $z\in M$ a pseudo-holomorphic curve in the symplectization $\mathbb{R}\times M$ such that the closure of its projection to $M$ contains $z$ (Lemma~\ref{lemma: main}). We then convert each of these projected curves into a $\partial$-strong $R_\lambda$-section (Corollary~\ref{corollary: embedded}). 
We then prove that there is a finite collection of $\partial$-strong $R_\lambda$-sections, with disjoint interiors, whose union intersects every orbit and such that each connected component of the complement fibers over~$\R$ (Lemma~\ref{lemma: rlsection}). 
This allows us to extend the finite collection of $\partial$-strong $R_\lambda$-sections to a foliation in the complement of a link formed by periodic orbits, thus completing a broken book decomposition. 
 
\subsection{Pseudo-holomorphic curves}
\label{sec: holomorphic}
 
In this section, we recall some facts concerning pseudo-holomorphic curves in symplectizations. 
We refer to Wendl's notes for a detailed introduction~\cite{We}. 

Fix a closed oriented 3-manifold~$M$ and a contact form $\lambda$ on~$M$ whose associated Reeb vector field~$R_\lambda$ is nondegenerate.
The symplectization is the 4-manifold~$\R\times M$, equipped with the symplectic form $d(e^s\lambda)$, where~$s$ denotes the additional real parameter. 
In this context a {\it pseudo-holomorphic curve} is a smooth map $u=F\to \R\times M$ from a Riemann surface $(F, j)$ satisfying $Tu\circ j=J\circ Tu$, where $J$ is an almost-complex structure compatible with the symplectic structure. 

We only consider pseudo-holomorphic curves that are {\it asymptotically cylindrical}. 
This means that $F$ is a closed compact surface with finitely many punctures which are of two types, called positive and negative. 
At every positive ({\it resp.} negative) puncture~$p\in F$, the curve $u$ is {\it asymptotic} to a closed orbit of~$R_\lambda$: 
this means that for some choice of local cylindrical coordinates~$(s,t)\in\R_+\times\mathbb{S}^1$ around~$p$ and for some periodic orbit~$\gamma$ of~$R_\lambda$, the surface $u(F)$ tends to $\{+\infty\}\times\gamma$ ({\it resp.} $\{-\infty\}\times\gamma$) in~$\R\times M$ as~ $s\to+\infty$.
In this situation, one says that~$\gamma$ is a {\it positive ({\it resp.} negative) end}, and that~$F$ {\it covers}~$\gamma$.
Composing with the projection $\pi:\R\times M\to M$, we see that $\pi\circ u(F)$ is a (possibly singular) surface in~$M$ bounded by some periodic orbits of~$R_\lambda$. 
More precisely, by the nondegeneracy assumption and a theorem of Hofer, Wysocki and Zehnder~\cite[Thm~2.8]{HWZ1} (see also~\cite[Thm~2.2]{Sie2}, \cite[Prop 3.2]{HT2} and~\cite[Thm 3.11]{We}), in some local cylindrical coordinates~$(r, \theta, z)$ adapted to~$R_\lambda$ around an end~$\gamma$, the projected curve~$\pi\circ u(F)$ is exponentially close to a {\it half-helix} of the form $\theta=f(z)$. 
 The direction $\{ \theta =f(z)\}$ is given by an eigenfunction of a nonzero eigenvalue of the {\it asymptotic operator} associated with $\gamma$, see \cite[Definition~3.5 and Theorem~3.11]{We}. 
 In particular, the asymptotic behaviour implies that the half-helix is $\partial$-strong along $\gamma$ (see Definition~\ref{defn-strong}).

The local picture of~$\pi\circ u(F)$ around that given Reeb periodic orbit is similar to one page of the open book decomposition shown in Figure~\ref{figure: radialcomponent}.
Note that near an elliptic periodic orbit, there is a well-defined germ of the bounded-time first-return map of the Reeb flow on the corresponding cylindrical ends.

The degree of the covering $\mathbb{S}^1\to\gamma$ induced by $t\mapsto \pi\circ\lim_{s\to+\infty}u(s,t)$ is called the {\it local multiplicity} at~$p$. 
It turns out to be positive at positive ends and negative at negative ends. 
Note that, after compactifying every puncture of~$F$ with a circle, several boundary components of~$F$ may be mapped onto the same periodic orbit of~$R_\lambda$. 
For a given periodic orbit of~$R_\lambda$, the sum of the local multiplicities of all positive ({\it resp.} negative) boundary components that are mapped on it is called the {\it global positive ({\it resp.} negative) multiplicity} of this orbit.

\subsection{Pseudo-holomorphic curves from embedded contact homology}\label{sec: ECH}
We now gather the results from embedded contact homology used in the proof of Theorem~\ref{thm: main}. 
We refer to~\cite{Hu} and~\cite{CHP} for a more detailed introduction to embedded contact homology.
The context is unchanged: $M$ is an oriented 3-manifold, $\lambda$ is a nondegenerate contact form, and~$R_\lambda$ is its Reeb vector~field. 

The ECH-chain complex~$ECC(M,\lambda)$ is generated over~$\Z_2$ (or $\Z$) by finite sets of simple periodic orbits of~$R_\lambda$ together with multiplicities. 
 
An ECH-holomophic curve between orbit sets $\Gamma$ and $\Gamma'$ is an asymptotically cylindrical pseudo-holomorphic curve such that $\Gamma$ is the limit (with multiplicities) of the positive punctures and $\Gamma'$ is the limit of the negative punctures.
The ECH-index of an ECH-holomorphic curve is an integer, whose definition can be found in Hutchings' notes~\cite[Def 3.5]{Hu}. 
The ECH-index-$1$ curves provide the differential for defining the complex~$ECH(M, \lambda)$, while the ECH-index-$2$ curves provide the $U$-map we are interested in.

The way the ECH-index-$1$ and $2$ curves approach their limit orbits is governed by the {\it partition conditions} \cite[Section 3.9]{Hu}. 
These relate the global and local multiplicities of the ends. 
For an elliptic limit periodic orbit, these conditions are rather cumbersome, but for a positive ({\it resp.} negative) hyperbolic limit periodic orbit, both its positive and negative global multiplicities are only allowed to be $0$ or~$1$ ({\it resp.} 0 or 2). 
This last condition is consistent with the way the ECH-index-$1$ or~$2$ curves involved in the definition of the differential or in the $U$-map break, see \cite[Section 5.4]{Hu}:
if a breaking involves a positive hyperbolic periodic orbit with multiplicity strictly larger than $1$, then there is an even number of ways to glue and these contributions algebraically cancel. 
 
Also, thanks to the case of equality in the writhe bound when the partition conditions are satisfied \cite[Lemma 5.1]{Hu}, near a hyperbolic limit periodic orbit, every orbit in its stable or unstable manifold does not intersect the corresponding positive or negative cylindrical end, see Figure~\ref{figure: ECH}.
This nice fact will not be used in the construction; it shows however that one might not expect our construction to provide supporting rational open books on the nose.

\begin{figure}[ht]
\includegraphics[width=.4\textwidth]{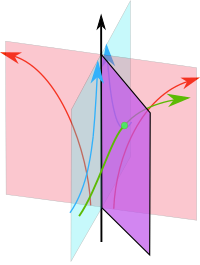}
\caption{The projection in~$M$ of an ECH-index $1$ or $2$ curve in~$\R\times M$, near an end which is a positive hyperbolic periodic orbit~$\gamma$ of~$R_\lambda$. 
The projection of the curve (purple) does not intersect the stable and unstable manifolds of~$\gamma$ (blue and red), so it lies in one of the four local sectors. 
It looks like an annulus bordered by~$\gamma$. 
It intersects all the orbits of~$R_\lambda$~(green) in its local sector. 
As suggested by the picture, its extension to the blown-up orbit~$\Sigma_\gamma$ is transverse to the extension of~$R_\lambda$ (as in Figure~\ref{figure: brokencomponent1} right), so it is a $\partial$-strong $R_\lambda$-section. 
The global positive multiplicity of~$\gamma$ is~$1$ and its global negative multiplicity is~$0$. 
If $\gamma$ was a negative periodic orbit, the picture would look similar, except there would be a second component of the projection of the curve in the opposite quadrant. 
}\label{figure: ECH}
\end{figure}

The map $U : ECC(M,\lambda)\to ECC(M,\lambda)$ is a degree $-2$ map counting pseudo-holomorphic curves passing through a point $(0,z)$ of the symplectization $\R \times M$ of $M$, where $z$ does not sit on a periodic orbit of $R_\lambda$. 
Now, there exists a finite orbit set $\Gamma =\Sigma_{i=1}^{k} \Gamma_i$ whose class $[\Gamma]$ in $ECH(M,\lambda)$ is such that $U([\Gamma])\neq 0$.
 Even though the map $U$ depends on the choice of~$z$ at the chain level, it does not at the homology one and so the fact that $U([\Gamma])\neq 0$ does not depend on $z$, see \cite[Section 3.8]{Hu}. 
The nonvanishing of the $U$-map is established via the naturality of the isomorphism between Heegaard Floer homology and embedded contact homology with respect to the $U$-map \cite{CGH0,CGH1,CGH2,CGH3} and the non-triviality of the $U$-map in Heegaard Floer homology \cite[Section 10]{OS}, or via the isomorphism with Seiberg-Witten Floer homology, as explained in~\cite{CHP}.

Recall that the {\it action}~$\mathcal{A}(\gamma)$ of an orbit $\gamma$ of $R_\lambda$ is the integral~$\int_\gamma \lambda$. 
The action of a collection of orbits is the sum of the actions of its elements, counted with multiplicities.
By the nondegeneracy assumption, there are only finitely many periodic orbits of action less than the action~$\mathcal{A} (\Gamma)$ of~$\Gamma$. 
 The {\it spectrum} of a class $[\Gamma] \in ECH(M,\lambda)$ is the minimal action of an orbit set representing $[\Gamma$].  

The main input from ECH-holomorphic curve theory is the following.

\begin{lemma}\label{lemma: main} 
 
Let $M$ be a compact 3-manifold and $\lambda$ a nondegenerate contact form on~$M$. 
For $\Gamma$ an ECH-cycle consisting of a sum of finite sets of periodic orbits of~$R_\lambda$ such that $U([\Gamma])\neq 0$, denote by~$\mathcal{P}$ the finite set of periodic orbits of the Reeb vector field~$R_\lambda$ of action less than $\mathcal{A} (\Gamma)$. 
Then
\begin{itemize}
\item
for every $z$ in $M\setminus \mathcal{P}$, there exists an embedded pseudo-holomorphic curve $u:F\to \R\times M$ asymptotic to periodic orbits of~$R_\lambda$ in $\P$ and such that the projection to~$M$ of some point in the interior of~$u$ is $z$;
\item
for every $z$ in~$\mathcal{P}$, there is a similar curve such that $z$ is either in the interior of the projection or in a boundary component of its closure. 
In the latter case, if the periodic orbit~$\gamma$ of~$R_\lambda$ containing~$z$ is positive ({\it resp.} negative) hyperbolic, the considered pseudo-holomorphic curve has global multiplicity~1 ({\it resp.}~2) along~$\gamma$ and it does not intersect locally the stable and unstable manifolds of~$\gamma$. 
\end{itemize}
 
\end{lemma}

\begin{proof} By definition of the $U$-map, for every generic $z\in M$, there is an ECH-index $2$ embedded curve in $\R\times M$ from $\mathcal{P}$ and passing through $(0,z)$.
Now, if $z$ is fixed, it is the limit of a sequence of generic points $(z_n)_{n\in \N}$. Through $(0,z_n)$ passes a pseudo-holomorphic curve~$u_n$ whose positive end is in $\mathcal{P}$. By compactness for pseudo-holomorphic curves in the ECH context, there is a subsequence of $(u_n)_{n\in\N}$ converging to a pseudo-holomorphic building, a component of which is an embedded pseudo-holomorphic curve through~$(0,z)$. All the asymptotics of the limit curves are in $\P$, since they all have action less than $\mathcal{A} (\Gamma)$. In particular, when $z$ is in $M\setminus \mathcal{P}$ it is contained in the interior of the projection of the curve to $M$. 

If $z$ is contained in one of the orbits of $\mathcal{P}$, it might be in a limit end of the curve and thus in the boundary of the closure of the projection of the curve to $M$. 
 
The second item then follows from the fact that positive ({\it resp.} negative) hyperbolic ends of ECH-index $1$ or $2$ curves may only have global multiplicity 0 or~1 ({\it resp.} 0 or~2) and they cannot intersect the stable and unstable manifolds around an end, as discussed previously. 
\end{proof}

Note that the compactness argument we refer to in the previous proof includes taking care of possibly unbounded genus and relative homology class, see \cite[Sections 3.8 and 5.3]{Hu}.

\subsection{From holomorphic curves to $\partial$-strong $R_\lambda$-sections}\label{subsection: sections}

Given a $\partial$-strong $R_\lambda$-section, the
local and global multiplicities of the boundary components are defined in the same way as for pseudo-holomorphic curves: if $i:S\to M$ is a $R_\lambda$-section and $c$ is a boundary component of $S$ mapped onto a periodic orbit~$\gamma$ of $R_\lambda$, the {\it local multiplicity} of $c$ is the degree of the induced cover $i:c\to\gamma$, and the {\it global multiplicity} of~$\gamma$ is the sum of the local multiplicities of the boundary components of~$S$ that are mapped on~$\gamma$.  
 
\begin{defn}\label{defn: link}
Let $S$ be a (not necessarily connected) $\partial$-strong $R_\lambda$-section.
\begin{itemize}
\item An orbit~$\gamma$ of~$R_\lambda$ is {\it asymptotically linking}~$S$ if for every $T\in\R$ the arcs~$\gamma([T, +\infty))$ and~$\gamma((-\infty, T])$ intersect $S$. 

\item If $\gamma$ is a periodic orbit in $\partial S$, consider its unit normal bundle~$\Sigma_\gamma$. The {\it self-linking} of $\gamma$ with~$S$ is the rotation number of the extension of~$R_\lambda$ to~$\Sigma_\gamma$, with respect to the 0-slope given by~$\partial_\gamma S$. 
\end{itemize}
\end{defn}

The most relevant case for us is when $\gamma$ is a hyperbolic periodic orbit: 
in this case the stable and unstable directions around~$\gamma$ locally determine four sectors. 
If the intersection of~$S$ with a tubular neighborhood of~$\gamma$ remains in one of these sectors, then the self-linking of~$S$ with~$\gamma$ is zero (as in Figure~\ref{figure: ECH}). 
This is the case of the ECH-index-1 or 2 curves given by Lemma~\ref{lemma: main}.

\medskip

We now promote the projected pseudo-holomorphic curves from Lemma~\ref{lemma: main} into $R_\lambda$-sections. 

\begin{prop}\label{corollary: embedded} 
 In the context of Lemma~\ref{lemma: main}, 
 for every $z$ in $M$ there exists a $\partial$-strong $R_\lambda$-section $S$ with boundary in $\P$ passing through $z$. 
Moreover, if $z$ is in~$M\setminus \P$, it is contained in the interior of $S$. 
Every positive ({\it resp.} negative) hyperbolic orbit~$k$ in $\partial S$ with self-linking number 0 has {\it local} multiplicity $1$ ({\it resp.}~$2$). 
\end{prop}

\begin{proof} Consider the point~$(0,z)$ in $\R\times M$. 
Denote by $\hat S_0$ the embedded pseudo-holomorphic curve from Lemma \ref{lemma: main} passing through~$(0,z)$.
It has a finite number of points where it is tangent to the holomorphic $\langle \partial_s, R_\lambda \rangle$-plane, where $s$ is the extra $\R$-coordinate. 
Indeed, Siefring proved that close enough to its limit end orbits in~$\P$, the pseudo-holomorphic curve is not tangent to this plane field~\cite[Theorem 2.2]{Sie2} and, by the isolated zero property for holomorphic maps, all these tangency points are isolated. 

Denote by $S_0$ the projection of $\hat S_0$ in~$M$. 
It contains $z$.
The projection of the tangency points correspond exactly to the points where $S_0$ is not immersed. 
We call these points the singular points of $S_0$ and we denote them by $x_i$, $i=1,\dots,p$.
Everywhere else $S_0$ is positively transverse to the Reeb vector field $R_\lambda$. 

 We apply the following desingularisation procedure to $S_0$, due to Fried for the part away from the singularities~\cite{Fr} . 

\begin{lemma}\label{lemma: desingularization} Let $R$ be a non-zero vector field on a $3$-manifold $M$ (here not necessarily Reeb) and $S_0$ be a compact surface with boundary in $M$, immersed away from
a finite number of interior points, whose boundary is tangent to $R$ and whose interior is positively transerse to $R$ away from its singularities.
 Assume that $S_0$ is $\partial$-strong with respect to~$R$.
For every $\epsilon>0$ and $x\in S_0$, consider the trajectory $\gamma_\epsilon (x)$ of $R$ through $x$ for time in $(-\epsilon, \epsilon)$.
Then there exists a $\partial$-strong $R$-section $S$ with $\partial S \subset \partial S_0$ and intersecting every $\gamma_\epsilon (x)$ for $x\in S_0$.
\end{lemma}
\begin{proof}
We modify $S_0$, first away from its singular points, and then around the singular points. 
First we put $S_0$ in general position by a generic perturbation, keeping it transverse to~$R$ away from the singular points: we make its self-intersections transverse and its triple self-intersection points isolated. 

Since the singular points are disjoint from the boundary components, we surround each singular point $x_i$, $i=1,\dots,p$, by a small ball~$B_i$ of the form of a flow box $D_i^2 \times [-1,1]$ that does not intersect $\partial S_0$, where the $[-1,1]$-direction is tangent to~$R$, so that the singular point $x_i$ is at the center and the boundary discs~$D_i^2\times \{\pm 1\}$ are disjoint from~$S_0$. 
Then $S_0$ only intersects~$\partial B_i$ along its vertical boundary $(\partial D_i^2)\times [-1,1]$. 

On $M\setminus (\cup_{i=1}^p B_i)$ the surface $S_0\setminus (\cup_{i=1}^p B_i)$ is immersed. 
We first treat the part~$S_0\setminus (\cup_{i=1}^p B_i)$ of~$S_0$. 
We blow-up the manifold $M$ along the boundary circles $\partial S_0$ to obtain a compact manifold $M_{\partial S_0}$ bounded by $2$-tori.
Since $S_0$ is $\partial$-strong, it extends to a compact surface still denoted $S_0$ on $M_{\partial S_0}$. 
By the $\partial$-strong property, $\partial S_0$ is transverse to the extension of $R$ along~$\partial M_{\partial S_0}$.
We can further assume by perturbing that $S_0$ only has transverse self-intersections even in~$\partial M_{\partial S_0}$.

The surface $S_0$ has a transversal given by $R$ so that we can resolve its self-intersections coherently to get an embedded surface~$S_1$ in $M_{\partial S_0}\setminus (\cup_{i=1}^p B_i)$, positively transverse to $R$, as in Figure~\ref{figure: FriedBord}. 
Each boundary torus of $M_{\partial S_0}$ has well-defined meridians: they correspond to the unit normal bundle of points in $\partial S_0$.
We can then isotope the resolved boundary in $\partial M_{\partial S_0}$, by an isotopy transversal to $R$, so that it becomes either transverse to the foliation by meridians or equal to a collection of meridian circles (see Figure~\ref{figure: dFriedSum}).
This isotopy can be extended to an isotopy of $S_0 \setminus (\cup_{i=1}^p B_i)$ transverse to $R$ and supported in a small neighborhood of~$\partial M_{\partial S_0}$.

\begin{figure}[ht]
\centering
\includegraphics[width=.5\textwidth]{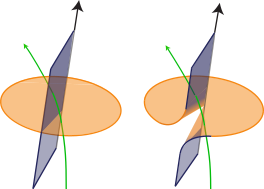}
\caption{How to resolve a line of intersections transversally to the flow.} 
\label{figure: FriedBord}
\end{figure}
 
\begin{figure}[ht]
\centering
\begin{picture}(102,40)(0,0)
\put(0,5){\includegraphics[width=.8\textwidth]{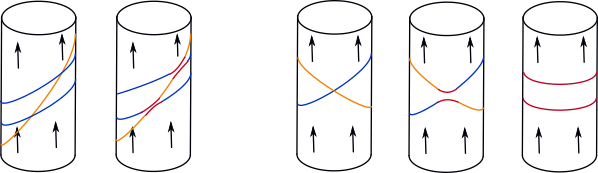}}
\put(5,0){(a)}
\put(24,0){(b)}
\put(55,0){(c)}
\put(74,0){(d)}
\put(93,0){(e)}
\end{picture}
\caption{ The blow-up~$\Sigma_\gamma$ of a boundary orbit~$\gamma$ in~$\partial S_0$. 
(a, c) The boundary~$\partial_\gamma S_0$ is a collection of curves transverse to~$R$ on~$\Sigma_\gamma$. 
(b, d) One resolves the double points of~$\partial_\gamma S_0$ transversally to~$R$. 
(b) If the obtained curves are transverse to the meridians, one stops here. 
(e) If the obtained curves are isotopic to meridians, one makes an isotopy of~$S_1$ and $\partial_\gamma S_1$ that turn them into meridians.} 
\label{figure: dFriedSum}
\end{figure}

Once this is done, the isotoped surface gives an immersed surface $S_1$ in~$M\setminus (\cup_{i=1}^p B_i)$ after blow-down. In case the intersection of the isotoped surface with a boundary component $\Sigma_\gamma$ of $\partial M_{\partial S_0}$ is along a meridian circle, the blow-down surface does not have $\gamma$ as a boundary component any more and crosses $\gamma$ transversally. The other components are $\partial$-strong.

We are left with the part inside the balls~$B_i$. 
The new surface $S_1$ now intersects every sphere $\partial B_i$ along an embedded collection of circles contained in the vertical part $(\partial D^2)\times [-1,1]$ and transverse to $R$, i.e. the $[-1,1]$, direction. 
We can extend $S_1$ inside the balls $B_i$ by an embedded collection of disks transverse to $R$. 
We get a surface $S$ which is a $\partial$-strong $R$-section. 
\end{proof}

The resolution of the self-intersections, described in the proof above, along a line of double-points ending in a boundary component is pictured in Figure~\ref{figure: FriedBord}: before and after the resolution. 
The resolution of triple points of intersection, coming generically from the transverse intersections of two branches of double points, are not an issue, since we can locally resolve one branch after another in any order and extend this resolution away, see Figure~\ref{figure: Desingularize}. 
Similarly, there might be lines of double points ending at the boundary of the balls $B_i$ surrounding singular points. We resolve them outside, up to the boundary of the balls. 
 
\begin{figure}[ht]
\centering
\includegraphics[width=.8\textwidth]{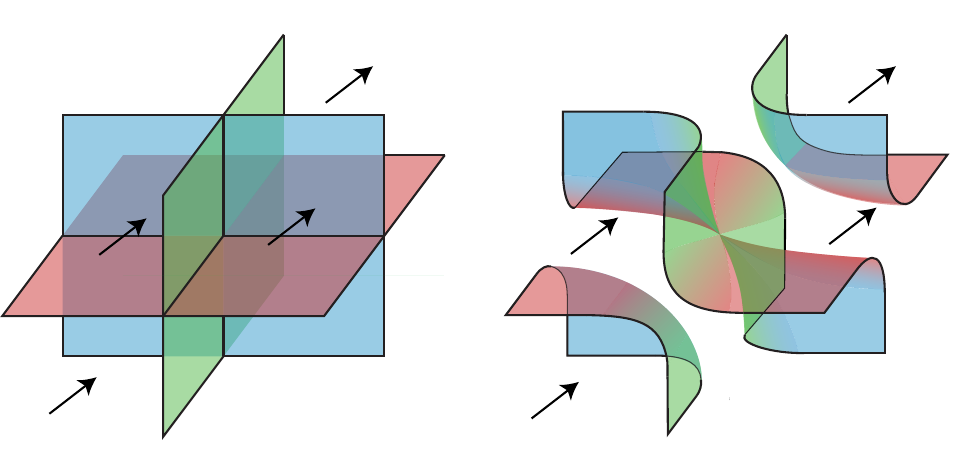}
\caption{How to resolve a triple point of self-intersections. One other way to picture what happens is to first resolve the intersection of the union of two surfaces and then add and resolve the intersections with a third one. }
\label{figure: Desingularize}
\end{figure}

In order to finish the proof of Proposition~\ref{corollary: embedded}, we apply Lemma \ref{lemma: desingularization} to~$S_0$ and $R_\lambda$ to get a surface $S$ and remark that it is easy to perform these surgery operations and keep the constraint of passing through the point $z$. 
Since~$S_0$ was $\partial$-strong as the projection of an ECH curve, then one checks that~$S$ can be taken to be $\partial$-strong .
 
Finally, along every positive or negative hyperbolic periodic orbit of $\partial {S}$ having self-linking number $0$ with~${S}$ (see Definition~\ref{defn: link}), the surface~$S$ is made of longitudinal annuli that do not intersect the stable and unstable directions. 
Since these annuli are embedded, their coordinates in a (longitude-meridian) basis are primitive, hence they can only have local degree $1$ if the orbit is positive hyperbolic or~$2$ if the orbit is negative hyperbolic. 
\end{proof}

Observe that the surface ${S}$ constructed above might have several connected components, but each connected component is a $\partial$-strong $R_\lambda$-section with boundary, because a closed surface cannot be transverse to a Reeb vector field.

\medskip

Recall that $\mathcal{P}$ denotes the set of those periodic orbits of~$R_\lambda$ whose action is smaller than~$\mathcal{A}(\Gamma)$. 
We now analyze what happens around the orbits of~$\mathcal{P}$.

\begin{lemma}\label{lemma: P}
In the context of Lemma~\ref{lemma: main}, let $k$ be a periodic orbit of~$R_\lambda$ that belongs to~$\mathcal{P}$. 
Then one of the following holds: 
\begin{enumerate}
\item there exists a $\partial$-strong $R_\lambda$-section transverse to~$k$;
\item there exists a $\partial$-strong $R_\lambda$-section containing $k$ in its boundary and whose self-linking with $k$ is non-zero;
\item there exists a $\partial$-strong $R_\lambda$-section containing $k$ in its boundary and whose self-linking with $k$ is zero. In this case $k$ is hyperbolic and each one of the sectors transversally delimited by the stable and unstable manifolds of~$k$ is intersected by at least one $\partial$-strong $R_\lambda$-section having $k$ as a boundary component.
\end{enumerate}
\end{lemma}

\begin{proof}
Let $z$ be a point of~$k$. 
Consider a small disc~$D$ containing~$z$ and transverse to~$k$, and take a sequence $z_n$ of generic points in~$D\setminus\{z\}$ converging to~$z$. 
In~$\R\times M$, take a sequence~$u_n$ of ECH-index 2 embedded pseudo-holomorphic curves through~$(0, z_n)$. 
In the limit pseudo-holomorphic building, there is a pseudo-holomorphic curve containing~$(0,z)$, whose projection~$S_0$ to~$M$ contains $z$. 
Then at least one of the following holds: 
\begin{itemize}
\item[(a)] $S_0$ is transverse to~$k$;
\item[(b)] $S_0$ contains $k$ in its boundary and its self-linking with $k$ is non-zero (this is the case when $k$ is an elliptic boundary orbit, but the partition conditions explained in Section~\ref{sec: ECH} imply that it cannot happen if $k$ is a hyperbolic boundary orbit);
\item[(c)] $S_0$ contains $k$ in its boundary and its self-linking with $k$ is zero.\end{itemize}

Observe that these are not disjoint cases: if $S_0$ is not embedded, a periodic orbit can be in the boundary of $S_0$ and intersect transversely $S_0$ in its interior.

Applying the resolution procedure of Lemma \ref{lemma: desingularization} to a surface~$S_0$ satisfying at least one of the first two cases, we obtain a~$\partial$-strong $R_\lambda$-section $S$. Then either $S$ is transverse to $k$ or $\partial S$ contains~$k$. In the later case the resolution procedure can change the self-linking number along $k$. If $k$ satisfies (a) and~(c) as boundary orbit of $S_0$, then $k$ will have non-zero self-linking number with respect to $S$. If $k$ satisfies (b) alone, or (a) and (b), then the self-linking number of $k$ with respect to $S$ will still be non-zero.

Now assume that $k\in \partial S_0$ satisfies (c) only. 
The limit curves satisfy the asymptotic conditions of \cite[Theorem 2.8]{HWZ1}, \cite{Sie1} \cite[Theorem 2.2]{Sie2}, \cite[Proposition 3.2]{HT2}, \cite[Theorem 3.11]{We}:
the order 1 expansion of the planar coordinate $z(s,t)$ (transverse to $k$ in $M$) says that along the boundary a limit curve is tangent (at the first order) to an annular half-helix. 
Then $k$ cannot be elliptic: in that case the self-linking of $S$ with $k$ would be irrational, which is absurd. 
Hence $k$ is hyperbolic and its tubular neighborhood has two or four sectors delimited by its stable and unstable manifolds. 
 From now on we assume that the points $z_n$ all belong to the same sector while converging to $z$. 

Now the transversality to the Reeb vector field implies that the half-helix has to stay in one sector. 
Indeed if it changes sector, in order to stay transverse to~$R_\lambda$, it can only change in one direction. 
So, in order to come back to the original sector after following $k$ for one longitude, it has to wrap around $k$: in this case the self-linking number would be positive, a contradiction. 

There are two cases depending on whether $k$ is at an intermediate level of the limit building with non-trivial positive and negative ends, or there is only a non-trivial positive and negative end (possibly followed by connectors). 

In the first case, for the limit building, the positive and negative ends at $k$ both have to follow a half-helix contained in a sector of $k$. If none of these sectors is the one containing the sequence $z_n$, then, close to the breaking curve, the curve has to enter and exit the sector containing $z_n$. 
We then look at its intersection with the stable and unstable manifolds delimiting the sector. 
This is a collection of circles that have to all be positively transverse to the Reeb vector field foliating the stable/unstable manifolds. 
In particular, none of the circles is contractible in the stable/unstable manifolds. 
The non-contractible circles have a fixed co-orientation given by the Reeb vector field: when the curve enters a sector along a stable/unstable manifold, it cannot exit along the same stable/unstable manifold nor along the other unstable/stable manifold.
 So this situation cannot happen: the limit curve has to be in the sector containing~$z_n$ near one end. 

In the second case, the limit building has only one non-trivial end on~$k$ (though there can be connectors) defining a half-helix (which is more a half-annulus) . 
The sequence of pseudo-holomorphic curves $u_n$ is converging to this half-helix and $k$ must be an end of the $u_n$'s. Now we have by the same argument as in the first case that $z_n$ has to be either on the same sector than the half-helix of $u_n$ or in the same sector than the one of the limit curve, otherwise we have a subsurface near the end of $u_n$ for $n$ large enough that is crossing a sector.

In both cases, we see that the limit building has a component with an end approaching $k$ by an annulus in the corresponding sector. 
\end{proof}

We can now assemble all the $\partial$-strong $R_\lambda$-sections together, by possibly changing some of them.


\begin{lemma}\label{lemma: rlsection}
In the context of Lemma~\ref{lemma: main}, there exists a finite number of\linebreak $\partial$-strong $R_\lambda$-sections with disjoint interiors, intersecting all orbits of $R_\lambda$. 

If an orbit of $R_\lambda$ is not asymptotically linking this collection of sections, it converges to one of their boundary components, which is a hyperbolic periodic orbit $k$ with local multiplicity $1$ or $2$.
In this case, each one of the sectors transversally delimited by the stable and unstable manifolds of~$k$ is intersected by at least one $\partial$-strong $R_\lambda$-section having the orbit as a boundary component.
\end{lemma}

\begin{proof} 
Start with a finite cover by flow-boxes of the complement of a small enough open neighborhood of $\P$. 
Through every point in a flow-box, Corollary~\ref{corollary: embedded} provides an embedded $\partial$-strong $R_\lambda$-section with boundary in the orbits of $\P$.
Since the closure of every flow-box is compact, there is a finite collection of surfaces intersecting every portion of orbit in the flow-box.
Using the desingularisation process described in Lemma~\ref{lemma: desingularization}, we can resolve the intersections between the $R_\lambda$-sections to obtain a finite collection of disjoint $\partial$-strong $R_\lambda$-sections~$S_1, \dots, S_n$ that intersect all the orbits in the complement of an open neighborhood of~$\P$ and whose boundary is contained in~$\P$.

We now analyse what happens in the neighborhood of~$\P$. 
Let $k$ be an orbit in~$\P$. 
If at least one of the surfaces $S_1, \dots, S_n$ is transverse to $k$, then by disjointness all are transverse to~$k$ (or do not even intersect~$k$), in which case there is nothing to do.

In the case where none of $S_1, \dots, S_n$ is transverse to~$k$, we apply Lemma~\ref{lemma: P} to~$k$. In the case of Lemma~\ref{lemma: P}~(2), we add the $\partial$-strong $R_\lambda$-section $S$ given by Lemma~\ref{lemma: P} to the collection~$S_1, \dots, S_n$ and we resolve once more the intersections of~$S\cup S_1\cup \dots\cup S_n$ using the common transverse direction $R_\lambda$. We obtain a new collection of $\partial$-strong $R_\lambda$-sections with disjoint interiors intersecting~$k$.

The remaining case, Lemma~\ref{lemma: P}~(3), arrises when $k$ is hyperbolic and the $\partial$-strong $R_\lambda$-sections given by Lemma~\ref{lemma: P} have self-linking equal to zero. 
Then the stable and unstable manifolds of the periodic orbit $k$ transversally delimit four sectors in a neighborhood of a point in the orbit, and Lemma~\ref{lemma: P} yields four surfaces (two if $k$ has negative eigenvalues) that cover all the sectors around~$k$. 
Again, we add these new $\partial$-strong $R_\lambda$-sections to our previous collection of $\partial$-strong $R_\lambda$-sections and then resolve the intersections of this new family using the common transverse direction $R_\lambda$. 
\end{proof}

Finally we have a $\partial$-strong $R_\lambda$-section $S$ (possibly disconnected), so that every orbit of $R_\lambda$ is either a boundary component or intersects $S$ transversely.
The union~$K$ of the boundary orbits of $S$ is a subset of~$\mathcal{P}$. 
If every orbit of~$K$ is asymptotically linking $S$, we get a rational open book. However, we can have boundary components where some orbits of $R_\lambda$ accumulate without intersecting the corresponding surface. 
These boundary components are necessarily hyperbolic periodic orbits and they all have local multiplicity one or two. 
We have now all the elements to prove Theorem~\ref{thm: main}.

\begin{proof}[Proof of Theorem~\ref{thm: main}] 
Lemma~\ref{lemma: rlsection} gives an $\partial$-strong $R_\lambda$-section~$S$ intersecting every orbit of the Reeb flow~$R_\lambda$, and we want to turn it into a broken book decomposition. 
Said differently, the $\partial$-strong $R_\lambda$-section $S$ forms a trivial lamination of~$M\setminus K$, and we have to extend $S$ into a foliation of~$M\setminus K$. 

For convenience, we first double all the components of~$S$ who have at least one boundary component on a hyperbolic periodic orbit and are not asymptotically linking the orbits in their stable/unstable manifolds. The two copies are separated in their interior by pushing along the flow of $R_\lambda$. We keep the notation $S$ for this new $\partial$-strong $R_\lambda$-section.
We then cut $M$ along~$S$ and delete standard Morse type neighborhoods of the hyperbolic orbits with self-linking number equal to zero of $\partial S$ as in Figure~\ref{figure: foliation}. 
 We denote by~$M_0$ the resulting manifold.

We claim that $M_0$ is a sutured manifold, foliated by compact $R_\lambda$-intervals. 
 Indeed, 
observe that when $R_\lambda$ is asymptotically linking~$S$ near an orbit~$k$ of~$\partial S$, the flow of $R_\lambda$ near $k$ has a well-defined first-return map on~$S$. These orbits are then decomposed by $S$ into compact segments.
When we are near a positive hyperbolic periodic orbit $k_b$ where the flow is not asymptotically linking with $S$, then~$S$ intersects a Morse type tubular neighborhood of $k_b$ in at least $8$ annuli, at least two in each sector (because of the doubling operation). Between two annuli in the same sector, the orbits of~$R_\lambda$ are locally going from one annulus to the other, thus an orbit is decomposed into compact intervals. 
If two consecutive annuli belong to different adjacent sectors, then they are cooriented by $R_\lambda$ and can be pushed in the direction of the invariant manifold of~$k_b$ separating them and glued to form an annulus transverse to~$R_\lambda$. 
Again every local orbit of $R_\lambda$ ends or starts in finite time on some (possibly glued) annulus.
 Thus $M_0$ is an $I$-bundle with oriented fibers, hence a product, and the conclusion follows. 

To finish the proof, we just have to glue back the Morse-type neighborhood of the hyperbolic orbits with self-linking number equal to 0 of~$\partial S$ to~$M_0$, that we foliate with the local model of a broken book, as on the bottom of Figure~\ref{figure: foliation}. The construction near a negative hyperbolic periodic orbits is analogous.
\end{proof}

\begin{figure}[ht]
\centering
\includegraphics[width=.8\textwidth]{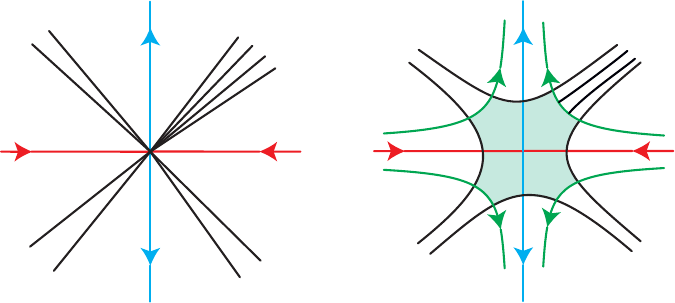}

\includegraphics[width=.78\textwidth]{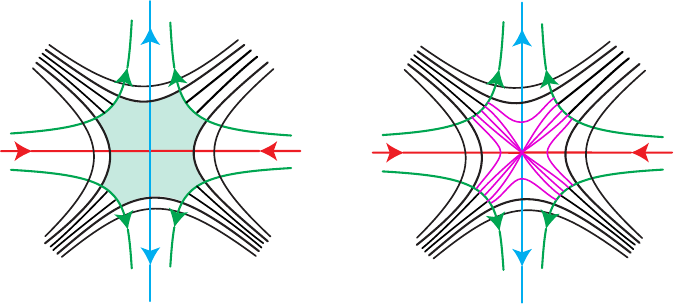}
\caption{A transversal view of a broken component $k_b$ of the binding. 
 On the top left, the $\partial$-strong $R_\lambda$-section~$S$ around~$k_b$ after the doubling procedure is shown (black). 
It is made of an even number of annuli---the point being to have at least~$2$---in each sector. 
 For building the foliation of~$M\setminus K$, one reconnects the first and last part of~$S$ in each sector in the neighborhood of~$k_b$, removes a smaller neighborhood (green), and cuts the resulting manifold along the (modified) $S$. 
The result is a trivial $I$-bundle (top right).
From the foliation of the trivial $I$-bundle (bottom left), one adds the neighborhood of~$k_b$ back, foliated with the local model of a broken binding orbit (bottom right).}
\label{figure: foliation}
\end{figure}

\begin{rmk} If $R_\lambda$ is supported by some open book decomposition, then every embedded projected pseudo-holomorphic curve gives rise to a new rational open book decomposition for $R_\lambda$ by the constructions above. Indeed, the self-linking numbers along the boundary components of the holomorphic curve all become positive after resolving the intersections with a page. 
In particular, the abundance of embedded projected pseudo-holomorphic curves given by the non-triviality of the $U$-map in embedded contact homology, or the differential, typically furnishes many open books for the same Reeb vector field.
\end{rmk}

\section{Applications}\label{sec-app}

The purpose of this section is to use broken book decompositions to study the dynamics of nondegenerate Reeb vector fields. 
The proofs are always divided into two cases: 
when the broken book decomposition is a rational open book decomposition ($K_b=\emptyset$), in which case one studies the diffeomorphism given by the first-return map on a page;
or when the broken binding is nonempty ($K_b\neq \emptyset$), in which case the stable and unstable manifolds of the broken binding components provide enough information to prove Theorems~\ref{thm: infinite} and~\ref{thm: entropy}. 

In Section \ref{ssec-int} we prove that these invariant manifolds always intersect. 
The case in which the intersections are transverse is simpler and the proof of Theorem \ref{thm: infinite} in this situation is explained in Section~\ref{ssec-strongly}. 
The proof without the transversality assumption requires considering several cases and destroying one-sided intersections. 
This proof is given in Section~\ref{ssec-nondeg}, the results there provide also a proof for Theorem~\ref{thm: entropy}. 
In Section~\ref{ssec-other} we discuss how to change a broken book decomposition with exactly one broken component into a rational open book decomposition supporting the same Reeb vector field, proving Theorem~\ref{thm: homoclinic}.
Finally in Section~\ref{ssec-improving}, we discuss the extension of Theorems \ref{thm: infinite}, \ref{thm: entropy} and \ref{thm: homoclinic} to an open neighborhood of the set of nondegenerate vector fields or strongly nondegenerate vector fields, accordingly.

\subsection{Flows on 3-manifolds: entropy and horseshoes}\label{ssec: horseshoe}

In this section we consider a general non-singular flow on a closed 3-manifold and recall certain aspects of the dynamics that we use in the proofs of Section~\ref{sec-app}. 

Let $\phi^t:M\to M$ be a flow on a compact 3-manifold. The flow needs to be $C^2$ for most of the results below. 
The topological entropy $h_{top}(\phi^t)$ is a non-negative number that measures the complexity of the flow. We review a definition, originally introduced by Bowen \cite{Bow70}. 
We first endow $M$ with a metric $d$. Given $T>0$ we define
$$d_T(x,y)=\max_{t\in [0,T]}d(\phi^t(x),\phi^t(y)),$$
for any pair of points $x,y\in M$. A subset $S\subset M$ is $(T,\epsilon)$-separated if for all $x\neq y $ in $S$ we have that $d_T(x,y)>\epsilon$. Let $N(T,\epsilon)$ be the maximal cardinality of a $(T,\epsilon)$-separated subset of $M$. The topological entropy is defined as
$$h_{top}(\phi^t)=\lim_{\epsilon\to 0}\limsup_{T\to \infty}\frac{1}{T}\log(N(T,\epsilon)).$$

A Smale horseshoe is a compact invariant set of the flow, in which the dynamics is semi-conjugate to the suspension of a shift of finite type, that is the prototypical example of chaotic dynamics, see for example Section~1.9 of the book~\cite{KH}. 
We ask further that the map realising the semi-conjugation is continuous with respect to the metric $d$ in $M$ and the natural metric in the shift-space. There is a strong relation between positive entropy and horseshoes (see for example \cite{LS19} for the case of flows or the original work of Katok \cite{Katok} for diffeomorphisms):

\begin{thm}\label{thm: horseshoe}
Let $\phi^t$ be a $C^2$-flow on a closed 3-manifold generated by a non-vanishing vector field. Then $h_{top}(\phi^t)>0$ holds if and only if there exists a Smale horseshoe subsystem of the flow.
\end{thm}

As a consequence, if the topological entropy is positive the number of hyperbolic periodic orbits grows exponentially with respect to the period (see Theorem 1.1 in \cite{LS19}). 
A Smale horseshoe persists under small perturbations, hence having positive entropy is an open condition for 3-dimensional flows.

A Smale horseshoe can be obtained from a transverse homoclinic intersection, see for example Theorem 6.5.5 in \cite{KH}.
Recall that a {\it homoclinic} orbit lies in the intersection of the stable and unstable manifolds of the same hyperbolic periodic orbit, while a {\it heteroclinic} orbit is an orbit that lies in the intersection of a stable manifold of a hyperbolic periodic orbit and an unstable manifold of another hyperbolic periodic orbit. 
From the above discussion one can conclude that having a transverse homoclinic intersection implies that the topological entropy of the flow is positive and that there are infinitely many periodic orbits. 

We use the term {\it homo/heteroclinic orbit} for an orbit that is either homoclinic or heteroclinic. 

\begin{defn}\label{defn: intersections}
Consider a flow $\phi^t$ and two hyperbolic periodic orbits (not necessarily distinct) with a homo/heteroclinic orbit connecting them. 
\begin{enumerate}
\item The orbits have a  {\it homo/heteroclinic connection} if the corresponding stable and unstable manifolds coincide.
\item The orbits have a {\it crossing homo/heteroclinic intersection} if the stable and unstable manifolds cross topologically, where crossing is in the topological sense of Burns and Weiss \cite{BW}. It means that there exists an embedded disk $D$ in $M$ transverse to the flow such that the intersection of $D$ with the stable and unstable manifolds contains two arcs $\gamma_{st}$ and $\gamma_{un}$ with boundaries in $\partial D$ and whose endpoints alternate in $\partial D$ (in particular they have to intersect and ``cross").
\item The orbits have a {\it one-sided homo/heteroclinic intersection}, or orbit, if the stable and unstable manifolds intersect and do not cross.
\end{enumerate}
\end{defn}

 Note that these definitions include the case where the stable and unstable manifolds intersect along an interval transverse to the flow and either cross or stay on one side at the boundary components.
 
 The transversality condition to obtain a Smale horseshoe can be weakened.
A heteroclinic cycle is a sequence of hyperbolic periodic orbits $k_0,k_1,\ldots, k_{n-1}, k_n=k_0$ such that the unstable manifold of $k_i$ intersects the stable manifold of $k_{i+1}$ for every $0\leq i\leq n-1$. In dimension 3, a heteroclinic cycle such that the intersection between the unstable manifold of $k_0$ and the stable manifold of $k_1$ is a crossing intersection and all other intersections are one-sided suffices to have positive topological entropy, see~\cite[Theorem 2.2]{BW}.

\subsection{Reeb vector fields supported by rational open books}\label{ssec: flux}

We consider the case in which a nondegenerate Reeb vector field $R_\lambda$ is supported by a rational open book decomposition and review known conclusions on the dynamics. We prove, at the end of the section, that such a vector field has two or infinitely many periodic orbits, see Theorem~\ref{thm: periodic}.

Let $(M,\lambda)$ be a contact closed 3-manifold such that the Reeb vector field~$R_\lambda$ is supported by a rational open book decomposition. 
Taking a page~$S$ of the open book, there is a well-defined homeomorphism $h:S\to S$ given by the first-return map along the flow. Moreover, the differential 2-form $d\lambda$ defines an exact area form on $S$ that is invariant under $h$. The map $h$ satisfies a stronger condition than being area-preserving. Consider the induced map 
$$h_*-I :H_1(S,\Z)\to H_1(S,\Z).$$

\begin{defn}\label{defn: fluxzero}
A homeomorphism $h:S\to S$ has {\it flux-zero} if for every curve $\gamma$ whose homology class is in $\ker (h_*-I)$, we have that ~$\gamma$ and $h(\gamma)$ cobound a $d\lambda$-area zero $2$-chain.
\end{defn}

Observe that this definition is valid for any exact area form. If $h$ is given by the first-return map of a Reeb vector field, it is clear that it is a flux-zero homeomorphism: take a curve $\gamma$ whose homology class is in $\ker (h_*-I)$, then $\gamma$ and $h(\gamma)$ bound a 2-chain in $S$ and also bound a surface tangent to the Reeb vector field in the ambient manifold $M$. 
The integral of $d\lambda$ on the union of the 2-chain and the surface has to be equal to zero by Stokes' theorem and is equal to the integral of $d\lambda$ on the 2-chain. 

The relation between flux-zero homeomorphisms and Reeb vector fields was studied in the paper \cite{CHL}. Given a surface $S$, possibly with boundary, and a homemorphism $h:S\to S$, we can consider the mapping torus $\Sigma(S,h)=S\times [0,1]/(x,0)\sim (h(x),1)$ that is a 3-manifold with boundary if $\partial S\neq\emptyset$. 
The vector field $\partial_t$, where $t$ is the $[0,1]$-coordinate, has as first-return map on $S\times\{0\}$ the map $h$. Observe that all the orbits of this flow are non-contractible, by construction.

The relation is given by Giroux's lemma~\cite[Lemma 2.3]{CHL} when $h$ is the identity near $\partial S$:

\begin{lemma}\label{lemma: fluxReeb}
Let $(S,d\lambda)$ be a compact oriented surface endowed with an exact area form and $h:S\to S$ a flux-zero homeomorphism that is the identity near $\partial S$. Then there exists a contact form $\alpha$ on $\Sigma(S,h)$ whose Reeb vector field $R_\alpha$ corresponds to $\partial_t$ and has ~$h$ as first-return map on $S\times \{0\}$.
\end{lemma}

\medskip

We now recall the Nielsen-Thurston classification~\cite{Th}. 

\begin{thm}\label{thm: NT}
If $S$ is a compact surface, possibly with boundary, and $h:S\to S$ a homeomorphism, then there is a map $h_0:S\to S$ isotopic to $h$ (not relative to the boundary) such that at least one of the following holds:
\begin{itemize}
\item[-] $h_0$ is periodic, meaning that some power of $h_0$ is the identity map;
\item[-] $h_0$ preserves some finite union of disjoint simple closed curves on $S$, in this case $h_0$ is called reducible;
\item[-] $h_0$ is pseudo-Anosov.
\end{itemize}
The three types are not mutually exclusive, but a pseudo-Anosov homeomorphism is neither periodic nor reducible. 
\end{thm}

When $h_0$ is reducible, one can cut $S$ along a finite collection of simple closed curves to obtain a homeomorphism of a smaller surface (possibly disconnected) and apply Theorem~\ref{thm: NT} to each of the pieces. Iterating this process, one obtains a map isotopic to $h$ that breaks up into periodic pieces and pseudo-Anosov pieces (see for example Section 13.3 in \cite{FaMa}). By abuse of notation, we call this final map $h_0$. 

In terms of the dynamics of the original map $h$, if $h_0$ has a pseudo-Anosov piece, then the map $h$ has positive topological entropy and infinitely many periodic orbits. We now use this classification to prove the following statement. 

While writing this paper, we learned about a more direct and general proof (for homeomorphisms, possibly degenerate) by Le Calvez~\cite{L}.

\begin{thm}\label{thm: periodic} Let $S$ be a compact surface with boundary different from the disk or the annulus and $\omega =d\beta$ an exact area form on $S$. 
If $h:S\to S$ is a nondegenerate area-preserving diffeomorphism with zero-flux then $h$ has infinitely many different periodic points.
\end{thm}

\begin{proof} 
By Giroux's Lemma~\ref{lemma: fluxReeb}, the zero-flux condition implies that $h$ can be realised as the first-return map of the flow of a Reeb vector field on~$\Sigma(S,h)$. 
Consider the Nielsen-Thurston decomposition of $h$. 
If it has a pseudo-Anosov component, then the conclusion of the theorem classically holds by Nielsen-Thurston theory, even without any conservative hypothesis. 
Otherwise, all the pieces of $h$ in the decomposition are periodic and there is a power $h^n$ of $h$ that is isotopic to the identity on each piece. 
This means that~$h^n$ is isotopic to a composition of Dehn twists on disjoint curves. 
Note that it is sufficient to show that $h^n$ has infinitely many periodic points.
 
If $h^n$ is isotopic to the identity, the conclusion is given by a theorem of Franks and Handel \cite{FH}, extended by Cristofaro-Gardiner, Hutchings and Pomerleano to fit exactly our case with boundary ~\cite[Proposition 5.1]{CHP}. 

If $h^n$ is not isotopic to the identity, consider the Nielsen-Thurston representative $h_0$ of $h^n$ given by a product of Dehn twists along disjoint annuli and at least one annulus $A$ is not boundary parallel. 
We now realise $h_0$ as the first-return map of a Reeb vector field~$R_0$ on $S$ in $\Sigma(S,h^n)$. 
This flow has no contractible periodic orbits since it is transverse to the fibration  over $\mathbb{S}^1$ and thus can be used to compute cylindrical contact homology. 
In the mapping torus of $A$, we have $\mathbb{S}^1$-families of periodic Reeb orbits realising infinitely many slopes in the suspended thickened torus. 
A generic perturbation transforms an $\mathbb{S}^1$-family into an elliptic and a positive hyperbolic periodic orbit. 
These all give generators for the cylindrical contact homology, since the other orbits (corresponding to periodic points of $h_0$) belong to different Nielsen classes. 
The invariance of cylindrical contact homology, see \cite{BH, HN}, suffices to conclude that the Reeb orbits given by the mapping torus of $h^n$ are at least the number given by the rank of the cylindrical contact homology computed with $R_0$, i.e. an infinite number.
\end{proof} 

We obtain two corollaries of Theorem~\ref{thm: periodic} for Reeb vector fields. Observe that the first-return map on the page $S$ is well-defined on the interior of $S$.
It might not extend smoothly to $\partial S$. In that case,
we filter the cylindrical contact homology complex by the intersection number of orbits with the page, that is equal to the period of the corresponding periodic points of~$h$. 
We can then modify the map $h$ near $\partial S$ to a zero-flux area-preserving diffeomorphism $h_k$ so that $(1)$ the modified monodromy~$h_k$ extends to $\partial S$, $(2)$ the orbits of period less than or equal to $k$ in the Nielsen classes not parallel to the boundary are not affected and $(3)$ $h_k$ is isotopic to~$h$ and so has the same Nielsen-Thurston representative $h_0$.
The arguments developed in the previous proof then apply to show the existence of periodic points of~$h_k$ with period bounded above by $k$, these are also periodic points of $h$ with period bounded by $k$. Hence $h$, has infinitely many periodic orbits.

\begin{cor}
Let $R_\lambda$ be a Reeb vector field on a closed 3-manifold $M$ that is supported by a rational open book decomposition. Let $S$ be a page and $h$ the first-return map to $S$. Assume that the Nielsen-Thurston decomposition of $h$ has a pseudo-Anosov component. Then $R_\lambda$ has positive entropy and infinitely many periodic orbits.
\end{cor}

\begin{cor}\label{cor: 2infty}
Let $R_\lambda$ be a Reeb vector field on a closed 3-manifold $M$ that is supported by a rational open book decomposition. Then $R_\lambda$ has two or infinitely many periodic orbits.
\end{cor}
 
\begin{proof}
If $S$ is a disk or an annulus the conclusion follows by Proposition 5.1 of \cite{CHP}: there are either 2 or infinitely many periodic orbits. 
In all other cases, Theorem~\ref{thm: periodic} provides infinitely many periodic orbits.
\end{proof}

\subsection{Heteroclinic and homoclinic intersections}\label{ssec-int}

In this section we study the hyperbolic orbits in the broken components of the binding of a supporting broken book, combined with the Reeb property of the flow. After completing this work, we realized that some arguments in this section were also used in~\cite{HWZ}.

For $k\in K_b$, a \emph{branch of the unstable manifold of~$k$} or \emph{unstable branch}, denoted by~$V^u(k)$, is a connected component of the unstable manifold of~$k$ cutted along $k$. 
It coincides with the unstable manifold minus~$k$ when $k$ is negative hyperbolic, but it consists of half of it when $k$ is positive hyperbolic. 
In the same way we use the notation $V^s(k)$ for a branch of the stable manifold of $k$.

\begin{lemma}\label{lemma: hyperbolic} Let $R_\lambda$ be a Reeb vector field for a contact form $\lambda$ carried by a broken book decomposition $(K,\F)$. 
Then, for every $k_0\in K_b$, every unstable branch $V^u(k_0)$ (\emph{resp.} stable branch $V^s(k_0)$) contains a homo/heteroclinic orbit asymptotic at $+\infty$ to a broken component of $K_b$. 
\end{lemma}

\begin{proof} 
An unstable branch $V^u(k_0)$ is made of an $\mathbb{S}^1$-family of orbits of $R_\lambda$, asymptotic to $k_0$ at $-\infty$. Each $\mathbb{S}^1$-family of orbits is a cylinder in $M$ with a boundary component in $k_0$, that is injectively immersed in $M$ since its portion near~$k_0$ for time $t<-T$, with $T$ large enough, is embedded.

We now argue by contradiction. 
Consider the finite collection of all the rigid pages $\mathcal{R} =\{ S_0 ,\dots , S_k\}$ of the broken book decomposition (see Definition~\ref{defn-rigid}). 
If no orbit in $V^u(k_0)$ limits to a broken component of $K_b$ at $+\infty$, then $V^u(k_0)$ has a return map on $\mathcal{R}$ which is well-defined. 
Then all the connected components of $V^u(k_0)\cap \mathcal{R}$ are compact and there are infinitely many of them, hence there is a rigid page, say~$S_0$, such that $V^u(k_0)\cap S_0$ has infinitely many connected components. 
Since $V^u(k_0)$ is an injectively immersed cylinder, the intersection with $S_0$ forms an infinite embedded collection~$C_0$ of closed curves in~$S_0$. 

Observe that $d\lambda$ is an area form on $S_0$. 
We claim that only a finite number of the curves in $C_0$ can be contractible in $S_0$. 
Two contractible components of~$C_0$ bound disks $D$ and $D'$ in $S_0$, these disks have the same $d\lambda$-area. 
Indeed~$D$ and $D'$ can be completed by an annular piece $A \subset V^u(k_0) $ tangent to $R_\lambda$ to form a sphere. 
Applying Stokes' theorem we obtain
\begin{equation}\label{eq-stokesdisks}
0= \int_{D\cup A\cup D'} d\lambda =\int_D d\lambda +\int_{A} d\lambda -\int_{D'} d\lambda = \int_D d\lambda -\int_{D'} d\lambda,
\end{equation}
because $d\lambda$ vanishes along $A$. 
Note that Equation~\eqref{eq-stokesdisks} also implies that $D$ is disjoint from $D'$, since $\partial D$ is disjoint from $\partial D'$. 
Since the total area of~$S_0$ is bounded, there are only finitely many contractible curves in $C_0$, as we wanted to prove. 

Thus infinitely many components of $C_0$ are not contractible in $S_0$, so at least two have to cobound an annulus $A'$ in $S_0$. 
The annulus $A'$ is transverse to $R_\lambda$ and its boundary components cobound by construction an annulus $A'' \subset V^u(k_0) $. 
We now apply Stokes' theorem to the torus $ A\cup A'' $ and get
$$0=\int_{A'\cup A''}d\lambda=\int_{A'} d\lambda>0,$$
a contradiction.

Hence each unstable (\emph{resp.} stable) branch of $k_0$ contains an orbit that is forward (\emph{resp.} backward) asymptotic to a component of $K_b$.
\end{proof}

Remark that in Equation~\eqref{eq-stokesdisks} the $d\lambda$-area of the disc $D$ is equal to the action of the periodic orbit $k_0$. 

We point out a consequence of the previous proof.

\begin{cor}\label{cor: hyperbolic}
Let $(K,\F)$ be a broken book decomposition with $K_b\neq \emptyset$ and $R_\lambda$ a Reeb vector field supported by $(K,\F)$. Then for every $k\in K_b$ and for every rigid page $S$, there is a natural number $n(k,S)$ that bounds the number of closed curves in the intersection of the stable/unstable branches of $k$ with~$S$. 
Similarly, every embedded cylinder longitudinally foliated by the flow and with one boundary component in $\mathcal{R}$ of action $\mathcal{A}>0$ intersects~$\mathcal{R}$ in less than a number of $n(\mathcal{A},\mathcal{R})$ circles. 
\end{cor}
The number $n(k,S)$ depends only on the action of $k$ and on the $d\lambda$-area and topology of the page $S$.

Recall that a heteroclinic orbit from~$k_0$ to~$k_1$ is an orbit contained in the unstable manifold of $k_0$ and in the stable manifold of $k_1$, that a {heteroclinic chain} based at~$k_0$ is a sequence of broken components $(k_0, k_1, \dots, k_{n-1},k_n)$ so that there is a heteroclinic orbit $\gamma_i$ from $k_i$ to $k_{i+1}$ for every $0\le i\le n{-}1$, and that a {homo/heteroclinic cycle} based at~$k_0$ is either a heteroclinic cycle based at~$k_0$ or the cycle~$(k_0, k_0)$ if there is a homoclinic orbit from~$k_0$ to~$k_0$. 

\begin{lemma}\label{lemma: cycle} 
There exists $k_0\in K_b$ such that 
\begin{itemize}
\item[$(a)$] either $k_0$ is positive hyperbolic and it has two homo/heteroclinic cycles $(k_0, k_1, \dots, ,k_{n-1},k_n=k_0)$ and $(k_0, k'_1, \dots, ,k'_{l-1},k'_l=k_0)$ based at~$k_0$ so that the orbits $\gamma_0$ and~$\gamma'_0$ are contained in different unstable branches of~$k_0$, 
\item[$(b)$] or $k_0$ is negative hyperbolic and it has one homo/heteroclinic cycle based at~$k_0$. 
\end{itemize}
\end{lemma}

\begin{proof} 
The argument is of graph-theoretical nature. 
Consider the directed graph~$G_b$ whose vertices are the broken components of the binding and such that there is one edge from $k_0$ to $k_1$ if there is a heteroclinic orbit between them. 
In particular $G_b$ may have loops (if there are homoclinic orbits) and double edges (if both unstable branches of~$k_0$ intersect one or both stable branches of~$k_1$). 

Recall that a strongly connected component of a directed graph is a subgraph where any two vertices may be connected by an oriented path and which is maximal with respect to this property. 
The quotient of the graph by the strongly connected components being acyclic and finite, there is a strongly connected component in~$G_b$ with no outgoing edge. 
We consider such a strongly connected component which is a sink, and denote it by~$G'_b$. 
Let $k_0$ be an orbit in~$G'_b$. 

If $k_0$ is a negative hyperbolic orbit, by Lemma~\ref{lemma: hyperbolic}, $k_0$ has at least one outgoing edge. 
Denote by~$k_1$ the end of this edge. 
If $k_1=k_0$ we are done. 
Otherwise, by strong connectivity there is a path in~$G'_b$ from $k_1$ to~$k_0$. 
Concatenating the edge $(k_0, k_1)$ with this path yields a heteroclinic cycle. 

If $k_0$ is positive hyperbolic, by Lemma~\ref{lemma: hyperbolic}, $k_0$ has at least two outgoing edges, corresponding to the two unstable branches of~$k_0$. 
Denote by $k_1$ and~$k_1'$ the respective ends of these two edges. 
Then by strong connectivity there are two paths in~$G'_b$ from $k_1$ to~$k_0$ and from $k_1'$ to~$k_0$ respectively. 
As before, concatenating the edges $(k_0, k_1)$ and $(k_0, k_1')$ with these paths yields the two desired homo/heteroclinic cycles. Observe that $k_1$ and $k_1'$ might be equal, but $\gamma_1$ and $\gamma_1'$ are different.
\end{proof}

\subsection{Theorem~\ref{thm: infinite} for strongly nondegenerate Reeb vector fields}\label{ssec-strongly}

We now prove Theorem~\ref{thm: infinite} when the Reeb vector field is {\it strongly nondegenerate}, that is when the invariant manifolds of the broken binding orbits intersect transversally. 
Then in Section~\ref{ssec-nondeg} we remove this transversality hypothesis and in Section~\ref{ssec-improving} we extend the proof to an open neighborhood. 

Observe that the strong nondegeneracy is a weaker hypothesis than being {\it Kupka-Smale}, since a Kupka-Smale vector field has in addition all its periodic orbits hyperbolic. 
The strongly nondegenerate condition is generic for vector fields due to \cite{Ku,Sm}. 
Another proof is due to Peixoto~\cite{Pe}. 
The latter extends to the Reeb context, so that strongly nondegenerate condition is also generic  for Reeb vector fields. 
It can also be derived from the results of Robinson \cite{Ro} which addresses the cases of symplectic diffeomorphisms and Hamiltonian vector fields.
All the needed perturbations are local and make no difference between area-preserving and flux-zero properties. 
Giroux's Lemma~\ref{lemma: fluxReeb} allows to convert local transverse area-preserving perturbations into Reeb perturbations. 

Theorem~\ref{thm: infinite} follows from the fact that if the broken book has broken components in its binding, then Lemma~\ref{lemma: cycle} implies that there are homo/heteroclinic cycles. 
The strongly nondegenerate hypothesis then gives a transverse homoclinic intersection, implying the existence of a Smale horseshoe. 
Theorem~\ref{thm: horseshoe} then implies that the flow has positive topological entropy and infinitely many periodic orbit.
 
If there are no broken components in the binding, the broken book is a rational open book. Then, Corollary~\ref{cor: 2infty} completes the proof of Theorem~\ref{thm: infinite} for strongly nondegenerate Reeb vector fields. 

\subsection{Theorems~\ref{thm: infinite} and \ref{thm: entropy} for nondegenerate Reeb vector fields}\label{ssec-nondeg}

We now prove Theorem~\ref{thm: infinite} in the case where we drop the strong nondegeneracy hypothesis, and obtain the result for nondegenerate Reeb vector fields.
Consider two hyperbolic periodic orbits (not necessarily different) with a homo/heteroclinic orbit connecting them. If the corresponding intersection of stable/unstable manifolds is not transverse then it is a connection, a crossing intersection or a one-sided intersection as in Definition~\ref{defn: intersections}.

After the proof in Section~\ref{ssec-strongly}, in order to prove Theorem~\ref{thm: infinite} for nondegenerate Reeb vector fields we need to consider only those supported by broken book decompositions with $K_b\neq \emptyset$. 
By Lemma~\ref{lemma: cycle} there is a heteroclinic cycle and as discussed in Section~\ref{ssec: horseshoe}, if at least one of the intersections in this cycle is a crossing intersection, then the vector field has positive topological entropy and infinitely many periodic orbits. 
Hence we consider the cases in which all the intersections are either connections or one-sided. 
We treat differently homo/heteroclinic connections and one-sided intersections because in the Reeb context, a homo/heteroclinic connection cannot be eliminated by a local perturbation of the Reeb vector field: 
one cannot displace a transverse circle from itself with a flux-zero map close to the identity.  

We first prove Theorem~\ref{thm: infinite} when all the homoclinic and heteroclinic intersections between the orbits of $K_b$ are connections (Lemma~\ref{lemma: full}).
We then consider the case of an unstable manifold that has only one-sided intersections with the stable manifolds of the orbits in $K_b$. 
The idea is to destroy the intersections, one by one, by perturbing the Reeb vector field and obtain a new Reeb vector field supported by the same broken book decomposition. 
For any given number $L>0$, the new Reeb vector field has a periodic orbit $k\in K_b$ whose unstable manifold intersects $L$ times the set of rigid pages along circles. This contradicts Corollary~\ref{cor: hyperbolic}. 
This is done in Proposition~\ref{lemma: consecutive}. 
Before that, we order the intersections along a stable or unstable manifold.

\medskip

We start with the case when there are only complete connections.

\begin{lemma}\label{lemma: full} 
Let $(K,\mathcal{F})$ be a broken book decomposition with $K_b\neq \emptyset$ supporting a nondegenerate Reeb vector field $R_\lambda$.
Assume that all the homoclinic or heteroclinic intersections between orbits in $K_b$ are connections. Then $R_\lambda$ has either two or infinitely many different simple periodic orbits.
\end{lemma}

\begin{proof} The idea is to cut $M$ along the stable/unstable manifolds of the broken components of the binding to obtain a manifold with torus boundary. 

Let $M'$ be the metric completion of $M$ minus the stable/unstable manifolds of the broken components of the binding. 
Then $M'$ is a 3-manifold with boundary and corners, possibly disconnected.

The boundary of $M'$ is made of copies of the stable and unstable manifolds of orbits in $K_b$ and has corners along the copies of orbits of $K_b$.
The Reeb vector field is tangent to the boundary, while the foliation $\F$ is now transverse to the boundary. The foliation $\mathcal{F}$ has two types of singularities, both along circles: it is singular along the radial components in $K_r$ and along the corners of $M'$ that corresponded to quadrants of a hyperbolic orbit in which $\mathcal{F}$ had a radial sector with nonempty interior. 
This configuration implies that (each connected component of) $M'\setminus K_r$ fibers over $\mathbb{S}^1$ and that the Reeb vector field has a first-return map defined on the interior of each one of these fibers. 

Let $N$ be a connected component of $M'$. If the fibers are neither disks, nor annuli, we obtain the existence of infinitely many periodic orbits using Theorem~\ref{thm: periodic}. Thus assume that in each connected component of $M'$ the fibers are either disks or annuli. 
Then the regular pages of $(K, \F)$ are disks or annuli and, by Lemma \ref{lemma: radial nonempty}, $K_r$ is nonempty.

A component $N$ containing an element of $K_r$ cannot have disk pages. 
So there is a component $N$ of $M'$ where all pages are annuli, having a boundary component on $k_r\in K_r$ contained in the interior of $N$ and the other one on the boundary of $N$. Note that this implies that $N$ is a solid torus. The boundary of $N$ is decomposed into the annuli given by the homoclinic or heteroclinic connections, each annulus being bounded by two (not necessarily different) components of $K_b$. 

The orbits of $R_\lambda$ define a 1-foliation on each annulus. Observe that no annulus is foliated by Reeb components of $R_\lambda$, since $R_\lambda$ is geodesible~\cite{Sullivan}: simply here, in the case of a Reeb component $d\lambda$ would be zero on the annulus, while the integral of $\lambda$ would be nonzero on the boundary, in contradiction with Stokes' theorem. Hence, the periodic orbits in the 2-torus~$\partial N$ turn in the same direction, are attracting on one side and repelling on the other side since we are alternately passing from a stable manifold to an unstable manifold. 

We now claim that we can change the fibration of the interior $N\setminus k_r$ by another fibration by annuli, close to the previous one (in terms of their tangent plane fields), so that it is still transverse to $R_\lambda$ in the interior and on the boundary. Indeed, outside of a neighborhood of $\partial N$, the Reeb vector field is away from the tangent plane field of the fibration by a fixed factor, in particular near $k_r$. 
Near $\partial N$, the Reeb vector field gets close to the tangent plane field of the fibration and is tangent to it along $K_b\cap \partial N$, but with a fixed direction since there are no Reeb components for the flow restricted to the boundary. Thus we can tilt the fibration in the other direction to make it everywhere transverse. This operation changes the slope by which the fibration approaches $k_r$ and the boundary $\partial N$.

Since the fibers are now everywhere transverse to $R_\lambda$, there is a well-defined first-return map that extends to the boundary to give an area-preserving homeomorphism of a closed annulus. Observe that this annulus is not necessarily embedded along $k_r\cup K_b$, but the map is well-defined. 
The boundary of any such annular page intersects at least one component of $K_b$, so that the first-return map to this annulus has at least one periodic point in the boundary. 
A theorem of Franks implies that there are infinitely many periodic points~\cite[Theorem~3.5]{Fra}.
\end{proof}

We now prove that an unstable manifold of an orbit in $K_b$ that does not coincide with the stable manifold of some orbit of $K_b$ must have a crossing intersection with some unstable manifold of an orbit in $K_b$.
 In order to prove such a statement we will destroy, one by one, one-sided intersections without creating new ones, by perturbing the Reeb vector field. To do so, we first have to order the intersections. 
Consider an orbit segment $\gamma$ that is not an entire periodic orbit. Thus $\gamma$ has no self-intersection and we say that $\gamma$ is a {\it simple orbit segment}. 

\begin{defn}\label{defn: length}
The \emph{length} of a simple orbit segment $\gamma$ is the number of components of $\gamma \setminus \mathcal{R}$, where $\mathcal{R}$ denotes the set of rigid pages. 
The length of a heteroclinic chain is the sum of the length of its components.
\end{defn}

Observe that the length of an orbit of the binding~$K$ is zero. 
The length of a full heteroclinic or homoclinic orbit is bounded, hence in $V^u(k)$ and~$V^s(k)$ for $k\in K_b$, the heteroclinic and homoclinic orbits are partially ordered. 

We also want to consider convergence of sequences of orbits.
For that we consider a small neighborhood $N(K_b)$ of $K_b$, made of the disjoint union of neighborhoods $N(k)$ of each $k\in K_b$. These neighborhoods are taken to be standard Morse type neighborhoods, as in Figure~\ref{figure: foliation}. 
Hence any orbit that enters and exits~$N(k)$ has to intersect at least one rigid page inside $N(k)$. 
If~$\gamma$ is a simple orbit segment, set $\hat{\gamma}=\gamma\setminus N(K_b)$ that is a collection of simple orbit segments. 

The first part of the following lemma is a tautology following from the definition of the length, and the second part follows by compactness.

\begin{lemma}\label{lemma: finiteness} Every simple orbit segment $\gamma$ of length equal to $L>0$ intersects the set of rigid pages $\mathcal{R}$ exactly $L-1$ times. 

For every $L>0$, there exists $N>0$ such that if $\gamma$ is a simple orbit segment of length bounded by $L$, then the action of $\hat{\gamma}$ is bounded from above by $N$. 
\end{lemma}

Next observe

\begin{lemma}[Compactness lemma]\label{lemma: compactification} Given $L>0$, the set of homo/heteroclinic orbits and simple orbit segments of length less than $L$ admits a natural compactification by chains of heteroclinic orbits and simple orbit segments whose length is at most~$L$.
\end{lemma}

\begin{proof}
Let $(\gamma_n)$ be a sequence of homo/heteroclinic orbits of length bounded by~$L$. 
Then every orbit $\gamma_n$ passes through less than $L$ components of $N(K_b)$ and we can extract a subsequence such that the orbits in the subsequence have the same pattern of crossings with $N(K_b)$. 
Then the collection of orbit segments $\hat{\gamma_n}$ 
have bounded action by Lemma \ref{lemma: finiteness} and are, up to extracting a subsequence, converging to a collection of orbit segments. 
Inside a component $N(k_1)$ of $N(K_b)$, the sequence $\gamma_n\cap N(k_1)$ has either bounded or unbounded action. 

In the first case, there is a subsequence that converges to a simple orbit segment in $N(k_1)\setminus k_1$. 

In the second case, there is a subsequence that converges to the concatenation of an orbit segment in $V^s(k_1)$, followed by~$k_1$ and then by an orbit segment in $V^u(k_1)$. 
Thus a subsequence of $(\gamma_n)$ converges to a heteroclinic chain. It is then immediate that the chain has length less than $L$ and Lemma~\ref{lemma: compactification} follows.

The proof for sequences of orbit segments is analogous.
\end{proof}

We can now prove the main result for one-sided intersections.

\begin{prop}\label{lemma: consecutive} Let $(K,\mathcal{F})$ be a broken book decomposition supporting a nondegenerate Reeb vector field $R_\lambda$. 
If an unstable branch $V^u(k)$ of some orbit $k \in K_b$ does not coincide with a stable manifold of an orbit in~$K_b$, then it contains a crossing intersection.
\end{prop}

\begin{proof} The proof of this result is not straightforward and will involve proving Lemma~\ref{lemma: eliminate}. 
The proof of Lemma~\ref{lemma: eliminate} uses Lemmas ~\ref{lemma: spiralling}, \ref{lemma: A} and~\ref{lemma: B}. 

We know by Lemma~\ref{lemma: hyperbolic} that $V^u(k)$ must intersect stable manifolds of periodic orbits in $K_b$.
We argue by contradiction and assume that $V^u(k)$ has no crossing intersection. Then $V^u(k)$ must contain only one-sided intersections. 
 
Recall that $V^u(k)$ is a cylinder bounded on one side by $k$ and which is foliated by orbits of~$R_\lambda$. 
We travel along $V^u(k)$, starting from~$k$. 
 
Consider the set $\mathcal{R}$ of rigid pages of the broken book decomposition. Then~$M\setminus \mathcal{R}$ is formed of product-type components. 
Since $R_\lambda$ is transverse to~$\mathcal{R}$, after leaving $k$, the unstable branch $V^u(k)$ enters successively some of these components along core circles $C_1, C_2, \dots$ of the cylinder $V^u(k)$, see Figure~\ref{figure: connectioncircles1}. 
Note that these entering circles are indeed ordered by the flow in $V^u(k)$. 
Also it exits each component along a circle that is an entering circle of the next component. 
This happens until the cylinder~$V^u(k)$ enters a component~$P$, along a circle $C_P$, that contains in its boundary an orbit~$k'\in K_b$ such that there is an orbit of $R_\lambda$ of a point in $C_P$ asymptotic to $k'$ that stays in~$P$. 
This means in particular that $V^u(k)$ and $V^s(k')$ intersect. 

\begin{figure}[ht]
\centering
\begin{picture}(75,62)(0,0)
\put(0,0){\includegraphics[width=.55\textwidth]{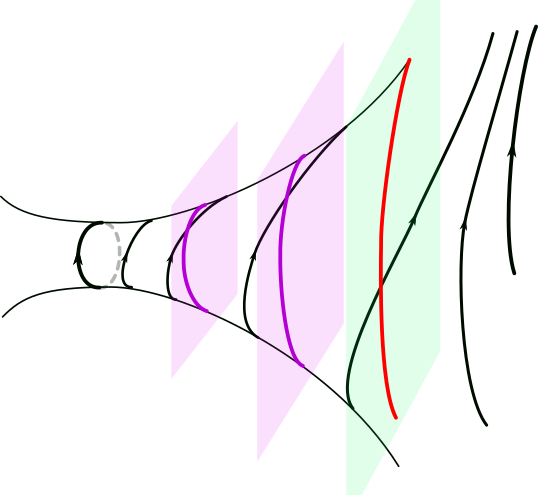}}
\put(7,30){$k$}
\put(67,40){$k'$}
\put(25,26){$C_1$}
\put(38,20){$C_P$}
\put(50,18){$C(k)$}
\put(27,40){$\mathcal{R}$}
\put(40,49){$\mathcal{R}$}
\put(53,59){$S$}
\end{picture}
\caption{
The unstable branch~$V^u(k)$ of some broken orbit, in case it coincides with no stable manifold (Proposition~\ref{lemma: consecutive}). 
The orbits of~$R_\lambda$ escape from~$k$. When travelling away from~$k$ one meets some rigid pages (purple) along circles $C_1, C_2, \dots$. 
Since there are connections in~$V^u(k)$, there is an orbit in the stable manifold~$V^s(k')$ for some broken orbit~$k'$. 
The circle $C_P$ is the last circle  in $V^u(k)\cap \mathcal{R}$. 
The region at the right of~$C_P$ is denoted by~$P$.
The surface $S$ is a regular page (green) of the broken book intersecting~$V^u(k)$ in~$P$, along a circle~$C(k)$ (red). 
}
\label{figure: connectioncircles1}
\end{figure}

Before arriving near~$k'$, the intersections of~$V^u(k)$ with the regular pages of $(K,\mathcal{F})$ also occur along circles. 
We pick a regular page~$S$ of~$(K,\mathcal{F})$ in~$P$. 
Then the 1-manifold $V^u(k) \cap S$ contains an embedded circle $C(k)$---the one after the last entrance circle $C_P$ in $P$---and $V^s (k')\cap S$ contains another embedded circle $C(k')$---the first intersection of $V^s(k')$ with~$S$ when flowing backward from~$k'$---so that the circles $C(k)$ and~$C(k')$ intersect along a nonempty compact set $\Delta$, containing only one-sided intersections. Indeed every point of $\Delta$ is located on an heteroclinic intersection from $k$ to $k'$, all of those being one-sided by assumption. 
The one-sided intersections can be on one side of~$C(k)$ or on the other, thus we further decompose $\Delta$ as the disjoint union of two compact sets $\Delta_+$ and $\Delta_-$, depending on the side of tangency, see Figure~\ref{figure: connectioncircles2}.
They could {\it a priori} be both nonempty. 
 
The circle~$C(k')$  is contained in a  branch of~$V^s(k')$ and the side $\Delta_+$ further determines a quadrant~$Q_+$ of $k'$ in which~$C(k)$ lies near~$\Delta_+$.

The idea of the rest of the proof is to destroy, inductively, the one-sided intersections of $V^u(k)$, starting from those passing through $\Delta$ and to find a new Reeb vector field supported by the same broken book decomposition, but such that $V^u(k)$ does not contain any heteroclinic orbit up to a certain {length}. 
The important consequence of Corollary~\ref{cor: hyperbolic}  is that the length up to which we need to eliminate the intersections is {\it a priori} bounded uniformly with respect to the features of the rigid pages (genus and area).  Let $L$ be a length that provides a contradiction to Corollary~\ref{cor: hyperbolic}.

\begin{lemma}\label{lemma: eliminate} If the component $V^u (k')$ is not a complete connection, i.e. it does not coincide with the stable manifold of an orbit in~$K_b$, we can slightly modify $R_\lambda$ in order to eliminate $\Delta_+$ and $\Delta_-$, without creating extra intersections of~$V^u(k)$ of length less than or equal to $L$.
\end{lemma}
\begin{proof}[Proof of Lemma \ref{lemma: eliminate}]

\begin{figure}[ht]
\centering
\begin{picture}(100,38)(0,0)
\put(0,0){\includegraphics[width=.8\textwidth]{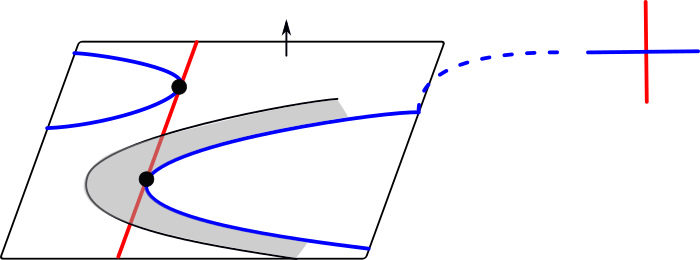}}
\put(22,34){$C(k)$}
\put(-3,19){$C(k')$}
\put(54,0){$C(k')$}
\put(27,24){$\Delta_-$}
\put(23,10){$\Delta_+$}
\put(17,11){$\mathcal{B}$}
\put(49,22){$A$}
\put(87,32){$Q_+$}
\put(43,33){$R_\lambda$}
\put(94,26){$k'$}
\put(68,26){$V^s(k')$}
\put(93.5,37){$V^u(k')$}
\end{picture}
\caption{
The surface~$S$ contains circles $C(k)=V^u(k)\cap S$ (red) and $C(k')=V^s(k')\cap S$ (blue), which are tangent to one another. 
Depending on the relative position at these contact points, they belong to~$\Delta_+$ or~$\Delta_-$. 
$\Delta_+$ determines a quadrant~$Q_+$ at~$k'$, as well as an unstable branch of~$V^u(k')$. 
An annulus~$A$ (grey) in~$Q_+\cap S$ is also shown. }
\label{figure: connectioncircles2}
\end{figure}

We now explain how to eliminate the intersections from $\Delta_+$. 
Recall that they determine near $k'$ a quadrant~$Q_+$ delimited by the branch of~$V^s(k')$ containing $\Delta_+$ and an unstable branch of~$V^u (k')$.

Let us first observe the following important property.

\begin{lemma}\label{lemma: spiralling} Every intersection from the half-unstable manifold $V^u(k')$ and any $V^s(k'')$, with $k'' \in K_b$, is non-crossing or complete, 
and on the other side of $Q_+$ (i.e. locally $V^s(k'')$ does not meet $Q_+$).
\end{lemma}
\begin{proof} We consider an annulus $A'$ transverse to $R_\lambda$ and which is a thickening of an essential curve in $V^u(k')$ near $k'$. 
If there was either a crossing intersection or a non-crossing intersection on the side of $Q^+$ between the half-unstable manifold $V^u(k')$ and some $V^s(k'')$, with $k'' \in K_b$, then $A' \cap Q_+ \cap V^s(k'')$ would contain an arc $\delta$ anchored in $V^u(k')$.
The half-infinite strip swept by the backward flow of $\delta$ would intersect $S$ into a set containing a curve spiraling to $C(k')$ on the side of $Q_+$, and thus there would be crossing intersections between $V^u(k)$ and $V^s(k'')$, contradicting the hypothesis on~$V^u(k)$.
\end{proof}

More generally, if $A'$ is an embedded annulus transverse to $R_\lambda$ which is a thickening of an essential closed curve $a$ in $V^s(k')$ near some $k' \in K_b$,
we say that some $V^s (k'')$, $k''\in K_b$, {\it spirals to $V^s(k')$} if $V^s(k'')\cap A'$ contains a sequence of arcs $a_n$ that is uniformly converging to the universal cover $\pi :\hat{a}\to a$ in $C^1_{loc}$-norm. This definition does not depend on the choice of $A'$.

Let $A=C(k') \times [0,\eta] \subset S$ be a small embedded annulus with $C(k')=C(k')\times \{0\}$, situated on the side of $Q_+$ in~$S$, see Figure~\ref{figure: connectioncircles2}.

\begin{lemma}\label{lemma: A} If $A$ is small enough, then its interior $C(k')\times (0,\eta)$ does not intersect any stable manifold of length at most~$L$ of any orbit $k''\in K_b$.
\end{lemma}
\begin{proof} Assume for a contradiction that for any~$A$ there is an intersection with a stable manifold and whose length is at most~$L$. 
Then we find a sequence $(x_n)$ of points in $A$ approaching $C(k')$, so that their orbits $(\gamma_n)$ are of length at most~$L$ and are asymptotic at $+\infty$ to elements in~$K_b$.
Using the Compactness Lemma~\ref{lemma: compactification}, up to extracting a subsequence, we can assume that $(x_n)$ converges to a point $x \in C(k')$ and $(\gamma_n)$ converges to a sequence of orbits $\delta_0,\dots,\delta_l$, where $\delta_0$ is a half-infinite orbit from $x$ to~$k'=k_0$, and, for $1\leq i\leq l$, $\delta_i$ is a heteroclinic orbit from~$k_{i-1} \in K_b$ to~$k_i \in K_b$.
Iterating the argument of the proof of Lemma \ref{lemma: spiralling}, we see that, following the sequence of connections by the backward flow from $k_l$ at each connection, for $i\leq l-1$, a portion of $V^s(k_l)$ spirals to the component $V^s(k_i)$ containing $\delta_i$ from the side containing $\gamma_n$ (for $n$ large enough). 
Note here that for the propagation of the spiraling property, there are two slightly different cases that can occur:
$V^s(k_l)$ spirals to the component $V^s(k_{i+1})$ and either $V^s(k_l)$ intersects $V^u(k_i)$ in which case it spirals to $V^s(k_i)$ by the proof of Lemma \ref{lemma: spiralling}, or it does not intersect $V^u(k_i)$ (because it spirals on the side of $V^s(k_{i+1})$ that is not intersecting $V^u(k_i)$ near $\delta_{i+1}$), but still since $V^s(k_{i+1})$ spirals to $V^s(k_i)$, so does $V^s(k_l)$ which spirals to $V^s(k_{i+1})$.

Thus finally a portion of $V^s(k_l)$ spirals to $V^s(k')$ from the side of $A$ and there are crossing intersections of $V^u(k)$ with $V^s(k_l)$.
\end{proof}

Thus, if $A$ is small enough, there is a collection of arcs in $C(k')\times \{\eta\}$ and a collection of arcs in $C(k)$ containing $\Delta_+$ that cobound a family~$\mathcal{B}$ of bigons in $A$ whose interior does not intersect any stable manifold of length at most~$L$ of any orbit $k''\in K_b$, see Figure~\ref{figure: connectioncircles2}.

\begin{lemma}\label{lemma: B} There exists a neighborhood $U$ of $\Delta_+$ in $S$ such that there is no orbit of length at most~$L$ from $U\cap \mathrm{int}(A)$
to the components of $U\setminus \mathcal{B}$.
\end{lemma}
\begin{proof} Arguing by contradiction, we can find as in the proof of Lemma~\ref{lemma: A} a sequence $(\gamma_n)$ of orbits of length at most~$L$, starting from $U\cap \mathrm{int}(A)$ and ending outside of $\mathcal{B}$, that converges to a sequence of orbits $\delta_0,\dots ,\delta_l$ such that~$\delta_0$ is an half-infinite orbit starting from a point $x$ of $\Delta_+$ and going to~$k'=k_0$; $\delta_i$ for $1\leq i\leq l-1$ is a heteroclinic orbit from $k_{i-1} \in K_b$ to $k_i\in K_b$; and~$\delta_l$ is a half-infinite orbit from $k$ to a point $y$ in $\Delta_+$.
With the same reasoning than in Lemma \ref{lemma: spiralling} and Lemma \ref{lemma: A}, we see that the backward flow of $C(k')\cap U$ has to come back spiraling along $C(k')$ on the side of $A$, thus creating crossing intersections with $V^u(k)$, a contradiction.
\end{proof}

The proof of Lemma \ref{lemma: eliminate} follows from the combination of Lemmas \ref{lemma: A} and \ref{lemma: B}.
First we choose the neighborhood $U$ from Lemma \ref{lemma: B} that does not intersect $\Delta_-$.

We consider a flux-zero area-preserving diffeomorphism $\psi$ supported in~$U$ that pushes $C(k)\cap U$ inside $(C(k')\times (0,\eta])\cap U$.
The modification of $R_\lambda$ 
will be done in a positive-flow-box product thickening of $U$ by modifying the direction of $R_\lambda$ so that the circle $C(k)$ entering the thickening of $U$ will be mapped to its image by $\psi$ when exiting, using Giroux' realization Lemma~\ref{lemma: fluxReeb}. 
This locally eliminates $\Delta_+$. Now, thanks to Lemmas \ref{lemma: A} and~\ref{lemma: B}, we see that we do not create extra connections of length $\leq L$ for~$k$. Indeed the modified part of $C(k)$ has no connections of length $\leq L$ that are not passing again in $U$ by Lemma \ref{lemma: A}; moreover they are not passing again through~$U\setminus \mathcal{B}$ by Lemma \ref{lemma: B}. If they enter again through $\mathcal{B}$ then they exit in $\mathrm{int}(A)\cap U$ and the same reasoning applies again.

We can apply the same argument to eliminate $\Delta_-$.
 \end{proof}

\medskip

Back to the proof of Proposition~\ref{lemma: consecutive}, we are left with the case in which the unstable component $V^u(k')$ is a complete connection to an orbit $k''$. 
We repeat the argument of Lemma~\ref{lemma: eliminate}: either the corresponding unstable manifold of $k''$ is a complete connection, or we can eliminate $\Delta_+$. The sequence of successive complete connections we are following is dictated at each step by the side which contains $\Delta_+$ and they all thus have a natural coorientation. In particular, we cannot use twice the same connection, since otherwise the corresponding cycle of complete connections would have, by tracing it back, to go through $V^s(k')$, which is not a complete connection, a contradiction. 
Since there are finitely many complete connections, this process stops and we can always eliminate $\Delta_\pm$.
Arguing by induction, we eliminate successively all intersections from $k$ of length less or equal to $L$, without creating new ones, and obtain a contradiction to Corollary~\ref{cor: hyperbolic} applied to the modified Reeb vector field for the unstable manifold $V^u(k)$. This terminates the proof of Proposition~\ref{lemma: consecutive}.
\end{proof}

We can now prove Theorems~\ref{thm: infinite} and \ref{thm: entropy} for nondegenerate Reeb flows. In view of Lemma~\ref{lemma: full} and Proposition~\ref{lemma: consecutive}, to prove Theorem~\ref{thm: infinite} we need to consider the case when there is at least one crossing intersection between the components of $K_b$. 

\begin{lemma}\label{lemma: cross} Let $(K,\mathcal{F})$ be a broken book decomposition supporting a nondegenerate Reeb vector field $R_\lambda$.
 Assume that there is at least one homo/heteroclinic intersection between broken components of the binding~$K_b$, with a crossing of stable and unstable manifolds. Then there is at least one homoclinic intersection with a crossing of stable and unstable manifolds, and thus positive topological entropy and infinitely many periodic orbits.
\end{lemma}

\begin{proof} If we have a homoclinic intersection with a crossing, as in the hypothesis, we are done. 
So assume that we have a heteroclinic intersection. 
We consider the set $\mathcal{C}$ of complete connections between components of $K_b$. As before cut $M$ along $\mathcal{C}$ to get a manifold $M'$ with boundary and corners. 
We let $K_b'$ be the collection of periodic orbits of $K_b$, when viewed in $M'$. 
The set $K_b'$ may contain several copies of the same orbit of $K_b$.

By hypothesis there is at least one heteroclinic orbit in $M'$ between elements of $K_b'$ along which there is a crossing intersection. 
Hence it is not in~$\partial M'$.
 Note that for every component $T$ of $\partial M'$,  the  total  number of stable and unstable manifolds of orbits $k_b \in T \subset \partial M'$  that are not themselves contained in $T$ is even, since there are as many stable than unstable manifolds in $T$ (every heteroclinic or homoclinic connection in  $T$  involves a stable and an unstable manifold).

Consider  a  connected component  $N$ of $M'$ that contains a crossing intersection, and let $k\in K_b$ be the orbit whose unstable manifold $V^u(k)$ is involved in this intersection. 
Following the heteroclinic intersections from~$V^u(k)$, as in Lemma~\ref{lemma: cycle}, we get a sequence of heteroclinic intersections. We claim that this sequence stays inside~$N$. 
Indeed, if it arrives to a periodic orbit in $K_b\cap \partial N$  along a stable manifold, then the two components of the unstable manifold of this periodic orbit are in $N$. 
We can thus construct a sequence such that all the stable and unstable manifolds involved are in~$N$. 
To see this, one can collapse  every component of $\partial N$ \color{black} to a generalized hyperbolic orbit with $2n$ stable/unstable manifolds:  we identify all the periodic orbits in a connected component of $\partial N$ and get rid of the stable and unstable branches in this boundary component. 
This is doable since there are no Reeb components for the Reeb flow on the boundary. 
We now have a closed manifold together with a Reeb vector field without connections. Lemma~\ref{lemma: hyperbolic} and Proposition~\ref{lemma: consecutive} imply that there is a cycle with only crossing intersections.
 
Near this cycle, we obtain a crossing homoclinic intersection, which is also a homoclinic intersection in $M$.
Positivity of topological entropy and the existence of infinitely many different periodic orbits comes from \cite[Theorem 2.2]{BW}.
\end{proof}

We now prove Theorem~\ref{thm: entropy}, for nondegenerate Reeb vector fields, stating that on a $3$-manifold that is not graphed, every nondegenerate Reeb vector field has positive topological entropy.

\begin{proof}[Proof of Theorem~\ref{thm: entropy} for nondegenerate Reeb flows] A nondegenerate Reeb vector field is carried by some broken book decomposition. If there is no broken component in the binding, then the broken book is in fact a rational open book. 
If~$M$ is not a graph manifold, then the monodromy of this rational open book must contain a pseudo-Anosov component in its Nielsen-Thurston decomposition. 
The first-return map of the Reeb vector field on a page is homotopic to the Nielsen-Thurston monodromy, so its topological entropy is bounded from below by the latter one, that is positive \cite[Expos\'e 10, IV]{FLP}.

If the binding of the broken book has broken components then all elements of $K_b$ that do not belong to complete connections contain, by Proposition~\ref{lemma: consecutive}, a crossing intersection. 
This proves the positivity of the entropy in this case.

If all stable and unstable manifolds of elements of $K_b$ are complete connections, then as in Lemma~\ref{lemma: full}, they decompose $M$ into partial open books and if $M$ is not graphed, then one of them must have some pseudo-Anosov monodromy piece in its Nielsen-Thurston decomposition and we obtain positive topological entropy.
\end{proof}

We end this section with a remark.
Note that the proof of Theorem~\ref{thm: periodic} gives, if $h$ is nondegenerate, the existence of positive hyperbolic periodic orbits, which are odd degree generators of cylindrical  contact  homology coming from the positive hyperbolic generators in the Morse-Bott families. 
Thus we are able to answer positively Question~1.8 of~\cite{CHP}:
\begin{cor}\label{cor: CHP} If $M$ is not $\mathbb{S}^3$ nor a lens space and if $R_\lambda$ is a strongly nondegenerate Reeb vector field on $M$, then it has a positive hyperbolic periodic orbit.
\end{cor}
Indeed, if there were none, $R_\lambda$ would have vanishing topological entropy~\cite{LS19} and would admit a supporting rational open book decomposition. 
Then Theorem~\ref{thm: periodic} would give at least one positive hyperbolic periodic orbit (among infinitely many other ones) in case every piece of~$h$ is periodic. In case there is a pseudo-Anosov piece, the existence of a Nielsen class with negative total Lefschetz index gives the same result and leads to a contradiction.

\subsection{Eliminating broken components}\label{ssec-other}

In this section we explore the possibility of obtaining a rational open book decomposition from a broken book decomposition. 
The results here can be applied to constructions by de Paulo and Salom\~ao \cite{dPS}, to obtain an open book decomposition. 
These results also give a proof of Theorem~\ref{thm: homoclinic}. 

In Lemma~\ref{lemma: cycle} we proved the existence of homo/heteroclinic cycles between the components of $K_b$. 
The interest of this result is that one could try to apply the local construction of Fried \cite{Fr} in the neighborhood of $(\cup_i k_i)\cup (\cup_i \gamma_i)\cup(\cup_i k_i')\cup (\cup_i \gamma_i')$ to get a surface of section $S_0$ that intersects transversally $k_0$ in its interior. 
We could then use $S_0$ to change the broken book decomposition and decrease the number of broken components of its binding. 
Iterating one would construct a supporting (up to orientations of the binding) rational open book. Unfortunately, Lemma~\ref{lemma: cycle} does not seem sufficient to make sure that $S_0$ intersects $k_0$, since the two heteroclinic cycles might involve twice the same stable or unstable branches (see Figures~\ref{figure: Nonos} and \ref{figure: Fried}).
However this works if there is only one broken component in the binding.

\begin{figure}[ht]
\centering
\includegraphics[width=.6\textwidth]{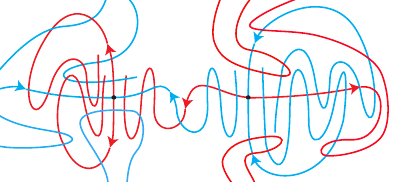}
\caption{Two hyperbolic orbits and their stable/unstable manifolds at which one cannot directly apply Fried's construction.}
\label{figure: Nonos}
\end{figure}

\begin{thm}\label{Fried}
 Let $R_\lambda$ be a strongly nondegenerate Reeb vector field for a contact form $\lambda$ carried by a broken book decomposition $(K,\F)$. Assume that $K$ contains at most one broken component. Then $R_\lambda$ has a Birkhoff section.
\end{thm}

\begin{proof}
Denote by~$k_0$ the broken component in the binding~$K$ and assume first that it is a positive hyperbolic periodic orbit. 
By the proof of Lemma~\ref{lemma: cycle}, each of the two branches of the unstable manifold of~$k_0$ intersect at least one stable branch of~$k_0$, and each of the two stable branches intersect at least one unstable branch. 
Therefore, up to a symmetry, there are two orbits $\gamma_{a}$ and $\gamma_{b}$ such that $\gamma_{a}$ belongs to both the west  unstable branch and the south  stable branch of~$k_0$, and $\gamma_{b}$ belongs to both the  east unstable manifold and the  north stable manifold of~$k_0$ (see Figure~\ref{figure: Fried}). 

Consider a small local transverse section~$D$ to~$R_\lambda$ at a point of $k_0$ and the induced first-return map~$f$. 
 
By taking a small transverse rectangle $r_W$ intersecting~$\gamma_a$ and considering its images by~$f$, the standard horseshoe construction~\cite[Section 6.5]{KH} implies that there is an iterate~$f^{k_W}$ such that $r_W$ and $f^{k_W}(r_W)$ intersect and the intersection contains a fixed point of $f^{r_W}$ that we denote by~$p_a$ (the pentagon on Figure~\ref{figure: Fried}).
Similarly there is a transverse rectangle $r_E$ intersecting~$\gamma_b$ and a suitable iterate~$f^{k_E}$ such that $f^{k_E}(r_E)$ intersects $r_E$ and the intersection contains a fixed point of $f^{k_E}$, denoted by~$p_b$. 
Denote by~$k_a$ and~$k_b$ the periodic orbits of~$R_\lambda$ corresponding to~$p_a$ and~$p_b$ respectively. 
More generally~\cite[Sections 1.9 and 6.5]{KH}, for every word~$w$ in the alphabet~$\{a, b\}$, one can find a periodic point~$p_w$ of~$f$ that follows $k_a$, and $k_b$ in the order given by~$w$. 
In particular there is periodic orbit~$k_{ab}$ that intersects~$D$ in two points~$p_{ab}$ and~$p_{ba}$ so that it remains close to~$\gamma_a$ for one period of~$k_a$ and close to~$\gamma_b$ for the next period of~$k_b$.

\begin{figure}[ht]
\centering
\begin{picture}(250,100)(0,0)
\put(0,0){\includegraphics[width=\textwidth]{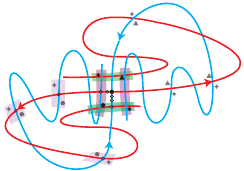}}
\put(52,30){$p_a$}
\put(50,38){$r_W$}
\put(25,35){$f(r_W)$}
\put(48,51.2){$p_{ab}$}
\put(68,29){$p_{ba}$}
\put(66,46){$p_{b}$}
\end{picture}
\caption{A transverse view of the orbit~$k_0$ (full black point in the center) and its stable/unstable manifolds (in blue/red respectively). 
The empty circle denote the orbit~$\gamma_a$ which sits at the intersection of the W-unstable branch and the S-stable branch of~$k_0$. 
Two small transverse rectangles $r_W$ and~$r_E$ in the W- and E-parts are shown in purple, together with their images by suitable iterates~$f^{k_W}$ and~$f^{k_E}$ of the first-return map~$f$, in green. 
In the intersection $r_W\cap f^{k_W}(r_W)$ lies a fixed point~$p_a$ of $f^{k_W}$, represented by a black pentagon, and its $f$-orbit is shown with grey pentagons. 
Similarly a fixed point~$p_b$ of $f^{k_E}$ is represented by a black triangle, and its $f$-orbit by grey triangles. 
Finally a fixed point $p_{ab}$ of $f^{k_W+k_E}$ is represented by a black 4-star, as well as its image $p_{ba}$ by~$f^{k_W}$. 
}
\label{figure: Fried}
\end{figure}

Now consider an arc connecting~$p_a$ to $p_{ab}$. 
When pushed by the flow, it describes a certain rectangle $R_1$ and comes back to an arc connecting~$p_a$ to~$p_{ba}$. 
Likewise an arc connecting~$p_b$ to~$p_{ba}$ describes a rectangle $R_2$ whose opposite side is an arc connecting $p_b$ to $p_{ab}$ (see Figure~\ref{figure: Fried3}).
Together these four arcs form a parallelogram~$P$ in~$D$ which contains $D\cap k_0$ in its interior. 
The union of $P$ and the two rectangles $R_1$ and $R_2$, forms an immersed topological pair of pants, which can be smoothed into a surface~$S_0$ transverse to~$R_\lambda$.
The main properties of~$S_0$ is that it is bounded by $k_a, k_b$ and $k_{ab}$, and it is transverse to~$k_0$. 

Now consider a page~$S$ of the foliation~$\F$ having $k_0$ in its boundary, with zero asymptotic linking number, take the union~$S_0\cup S$, and use the flow direction~$R_\lambda$ to desingularize the arcs and circles of intersection as in Lemma~\ref{lemma: desingularization} (see also Figure~\ref{figure: FriedBord}). 
The obtained surface intersects~ any small enough tubular neighborhood of $k_0$ along one meridian plus one or two longitudes. 
Therefore~$k_0$ is not anymore in the broken part of the binding, but is part of the boundary of the new surface. 
 Also the surface~$F_0$ intersects $k_a, k_b$, and~$k_{ab}$, so that these periodic orbits link positively the union~$F_0\cup S$. 
The resulting surface is then a genuine (rational) Birkhoff section for the Reeb vector field. We can thus obtain an open book decomposition from it adapted to the Reeb vector field. Observe that the orbits~$k_0, k_a, k_b$, and $k_{ab}$ are boundary components of radial type with respect to the new foliation.

The case where $k_0$ is a negative hyperbolic periodic orbit is treated in the same manner. The difference is that now one needs to consider the second iterate of the return map to a local transversal to the periodic orbit in order to have Figure~\ref{figure: Fried}. 
\end{proof}

\begin{figure}[ht]
\centering
\begin{picture}(40,50)(0,0)
\put(0,0){\includegraphics[width=.3\textwidth]{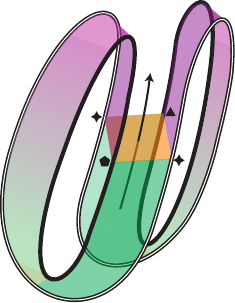}}
\put(6,30){$R_1$}
\put(37,30){$R_2$}
\put(22,24){$P$}
\put(23,38){$k_0$}
\end{picture}
\caption{The union of the rectangles $R_1$ and $R_2$ (whose color change from purple to green along the flow) with the parallellogram~$P$ (orange) which lies in~$D$ yields a topological pair of pants. It can be smoothed into a surface transverse to the Reeb flow, bounded by the periodic orbits~$k_a, k_b$ and~$k_{ab}$, and whose interior intersects~$k_0$. 
}
\label{figure: Fried3}
\end{figure}

\begin{conj}\label{conj-obd} Every Reeb vector field has a Birkhoff section.
\end{conj}

It follows from the previous considerations that if a strongly nondegenerate Reeb vector field has no homoclinic orbit, then any of its supporting broken book decomposition is in fact a rational open book, providing a proof of Theorem~\ref{thm: homoclinic}. Moreover, if a strongly nondegenerate Reeb vector field has at most one periodic orbit having a heteroclinic cycle then, by Lemma~\ref{lemma: hyperbolic}, it has a supporting broken book with at most one broken binding component, and by Theorem~\ref{Fried}, a Birkhoff section.

\subsection{Openness of the various properties}\label{ssec-improving}
In this last section, we discuss the nondegeneracy hypothesis and observe that our Theorems \ref{thm: main}, \ref{thm: infinite}, \ref{thm: entropy} and \ref{thm: homoclinic} hold for an open set of contact forms containing nondegenerate ones or strongly nondegenerate ones, accordingly.

Indeed, the arguments only use the fact that the periodic orbits in the binding are nondegenerate  and do not care about nondegeneracy of ``long" periodic orbits. The binding orbits are a subset of the set of orbits $\mathcal{P}$ for a nondegenerate contact form $\lambda$ that shows up in Lemma \ref{lemma: main}.
The set $\mathcal{P}$ contains all periodic orbits of $R_\lambda$ whose actions are less than $\mathcal{A} (\Gamma)$, where~$\Gamma$ is a representative of a homology class whose image by the $U$-map does not vanish. 
As we will see below, the action of an orbit set realizing a class which is not in the kernel of $U$ can be {\it a priori} bounded in a neighborhood of~$\lambda$, a manifestation of the continuity of the spectral invariants.

Precisely, there exists a $C^2$-neighborhood $N(\lambda)$ of $\lambda$ such that for every nondegenerate contact form $\lambda'$ in $N(\lambda)$,
there is a representative $\Gamma'$ for $\lambda'$ of a nonzero class in $ECH(M,\lambda')$ whose image under the $U$-map is nonzero and whose total action is {\it a priori} bounded by some $L>0$, depending only on $N(\lambda)$.
This is a consequence of the existence of the cobordism maps in $ECH$ (defined through Seiberg-Witten homology) given in
\cite[Theorem 2.3]{Hu2}.
We now shrink $N(\lambda)$ so that moreover for every form $\lambda'$ in~$N(\lambda)$, the periodic Reeb orbits of $\lambda'$ of action less than $L$, forming a set~$\mathcal{P}'$, are all nondegenerate -- note this is an open condition.
Now, if $\lambda'$ is a contact form in $N(\lambda)$, possibly degenerate, first its periodic  Reeb orbits of action less than $L$ are nondegenerate and second it can be approximated by a sequence of nondegenerate contact forms $\lambda'_n$ in $N(\lambda)$ whose periodic orbits of actions less than $L$ coincide with those of $\lambda'$. 
Since $\lambda_n'$ is nondegenerate, the conclusion of Lemma \ref{lemma: main} holds for $\lambda_n'$ and $\mathcal{P'}$: through every point~$z$ in $M$, one can find a projected holomorphic curve through $z$ and with asymptotics in $\mathcal{P'}$. 
By compactness for pseudo-holomorphic curves in the ECH context \cite[Sections 3.8 and 5.3]{Hu}, this property also holds for $\lambda'$ and~$\mathcal{P}'$. 
This is all we need to find a supporting broken book for $\lambda'$, with binding in the nondegenerate set $\mathcal{P}'$. 
The rest of the arguments to prove Theorems~\ref{thm: main}, \ref{thm: infinite}, \ref{thm: entropy} and \ref{thm: homoclinic} then carry over to $\lambda'$.

\end{document}